\documentclass[12pt,titlepage]{article}
\usepackage{latexsym,amsmath,amsthm}
\usepackage{amssymb}
\usepackage[margin=1in]{geometry}
\usepackage{graphicx}

\newtheorem{thm}{Theorem}[section]
\newtheorem{defn}[thm]{Definition}
\newtheorem{lemma}[thm]{Lemma}
\newtheorem{prop}[thm]{Proposition}

\newtheorem{conj}[thm]{Conjecture}

\newtheorem{question}[thm]{Question}

\newtheorem{property}[thm]{Property}

\newtheorem{exercise}[thm]{Exercise}

\title{Three tutorial lectures on entropy and counting\footnote{These notes were prepared to accompany a series of tutorial lectures given by the author at the 1st Lake Michigan Workshop on Combinatorics and Graph Theory, held at Western Michigan University on March 15--16 2014.}}

\author{David Galvin\thanks{dgalvin1@nd.edu; Department of Mathematics,
University of Notre Dame, Notre Dame IN 46556. Supported in part by National Security Agency grant H98230-13-1-0248.}}

\date{1st Lake Michigan Workshop on Combinatorics\\ and Graph Theory, March 15--16 2014}

\begin{document}

\maketitle

\begin{abstract}
We explain the notion of the {\em entropy} of a discrete random variable, and derive some of its basic properties. We then show through examples how entropy can be useful as a combinatorial enumeration tool. We end with a few open questions.

%
%
\end{abstract}

\tableofcontents

\section{Introduction}

One of the concerns of information theory is the efficient encoding of complicated sets by simpler ones (for example, encoding all possible messages that might be sent along a channel, by as small as possible a collection of 0-1 vectors). Since encoding requires injecting the complicated set into the simpler one, and efficiency demands that the injection be close to a bijection, it is hardly surprising that ideas from information theory can be useful in combinatorial enumeration problems.

These notes, which were prepared to accompany a series of tutorial lectures given at the 1st Lake Michigan Workshop on Combinatorics and Graph Theory, aim to introduce the information-theoretic notion of the {\em entropy} of a discrete random variable, derive its basic properties, and show how it can be used as a tool for estimating the size of combinatorially defined sets.

The entropy of a random variable is essentially a measure of its degree of randomness, and was introduced by Claude Shannon in 1948. The key property of Shannon's entropy that makes it useful as an enumeration tool is that over all random variables that take on at most $n$ values with positive probability, the ones with the largest entropy are those which are uniform on their ranges, and these random variables have entropy exactly $\log_2 n$. So if ${\mathcal C}$ is a set, and $X$ is a uniformly randomly selected element of ${\mathcal C}$, then anything that can be said about the entropy of $X$ immediately translates into something about $|{\mathcal C}|$. Exploiting this idea to estimate sizes of sets goes back at least to a 1963 paper of Erd\H{o}s and R\'enyi \cite{ErdosRenyi}, and there has been an explosion of results in the last decade or so (see Section \ref{sec-other papers}).

In some cases, entropy provides a short route to an already known result. This is the case with three of our quick examples from Section \ref{sec-quick applications}, and also with two of our major examples, Radhakrishnan's proof of Br\'egman's theorem on the maximum permanent of a 0-1 matrix with fixed row sums (Section \ref{sec-bregman}), and Friedgut and Kahn's determination of the maximum number of copies of a fixed graph that can appear in another graph on a fixed number of edges (Section \ref{sec-embedding}). But entropy has also been successfully used to obtain new results. This is the case with one of our quick examples from Section \ref{sec-quick applications}, and also with the last of our major examples, Galvin and Tetali's tight upper bound of the number of homomorphisms to a fixed graph admitted by a regular bipartite graph (Section \ref{sec-counting colorings}, generalizing an earlier special case, independent sets, proved using entropy by Kahn). Only recently has a non-entropy approach for this latter example been found.

In Section \ref{sec-basics of entropy} we define, motivate and derive the basic properties of entropy. Section \ref{sec-quick applications} presents four quick applications, while three more substantial applications are given in Sections \ref{sec-embedding}, \ref{sec-bregman} and \ref{sec-counting colorings}. Section \ref{sec-open problems} presents some open questions that are of particular interest to the author, and Section \ref{sec-other papers} gives a brief bibliographic survey of some of the uses of entropy in combinatorics.

The author learned of many of the examples that will be presented from the lovely 2003 survey paper by Radhakrishnan \cite{Radhakrishnan2}.

\section{The basic notions of entropy} \label{sec-basics of entropy}

\subsection{Definition of entropy} \label{subsec-definition of entropy}

Throughout, $X$, $Y$, $Z$ etc. will be discrete random variables (actually, random variables taking only finitely many values), always considered relative to the same probability space. We write $p(x)$ for $\Pr(\{X=x\})$. For any event $E$ we write $p(x|E)$ for $\Pr(\{X=x\}|E)$, and we write $p(x|y)$ for $\Pr(\{X=x\}|\{Y=y\})$.

\begin{defn} \label{defn-entropy}
The {\em entropy} $H(X)$ of $X$ is given by
$$
H(X) = \sum_x - p(x)\log p(x),
$$
where $x$ varies over the range of $X$.
\end{defn}
Here and everywhere we adopt the convention that $0\log 0 = 0$, and that the logarithm is always base 2.

Entropy was introduced by Claude Shannon in 1948 \cite{Shannon}, as a measure of the expected amount of information contained in a realization of $X$. It is somewhat analogous to the notion of entropy from thermodynamics and statistical physics, but there is no perfect correspondence between the two notions. (Legend has it that the name ``entropy'' was applied to Shannon's notion by von Neumann, who was inspired by the similarity to physical entropy. The following recollection of Shannon was reported in \cite{TribusMcIrvine}: ``My greatest concern was what to call it. I thought of calling it `information', but the word was overly used, so I decided to call it `uncertainty'. When I discussed it with John von Neumann, he had a better idea. Von Neumann told me, `You should call it entropy, for two reasons. In the first place your uncertainty function has been used in statistical mechanics under that name, so it already has a name. In the second place, and more important, nobody knows what entropy really is, so in a debate you will always have the advantage'.'')

In the present context, it is most helpful to (informally) think of entropy as a measure of the expected amount of surprise evinced by a realization of $X$, or as a measure of the degree of randomness of $X$. A motivation for this way of thinking is the following: let $S$ be a function that measures the surprise evinced by observing an event occurring in a probability space. It's reasonable to assume that the surprise associated with an event depends only on the probability of the event, so that $S:[0,1] \rightarrow {\mathbb R}^+$ (with $S(p)$ being the surprise associated with seeing an event that occurs with probability $p$).

There are a number of conditions that we might reasonably impose on $S$:
\begin{enumerate}
\item $S(1)=0$ (there is no surprise on seeing a certain event);
\item If $p < q$, then $S(p)>S(q)$ (rarer events are more surprising);
\item $S$ varies continuously with $p$;
\item $S(pq)=S(p)+S(q)$ (to motivate this imposition, consider two independent events $E$ and $F$ with $\Pr(E)=p$ and $\Pr(F)=q$. The surprise on seeing $E \cap F$ (which is $S(pq)$) might reasonable be taken to be the surprise on seeing $E$ (which is $S(p)$) plus the remaining surprise on seeing $F$, given that $E$ has been seen (which should be $S(q)$, since $E$ and $F$ are independent);  and
\item $S(1/2)=1$ (a normalizing condition).
\end{enumerate}

\begin{prop} \label{prop-uniqueness of surprise}
The unique function $S$ that satisfies conditions 1 through 5 above is $S(p) = -\log p$
\end{prop}

The author first saw this proposition in Ross's undergraduate textbook \cite{Ross}, but it is undoubtedly older than this.

\begin{exercise} \label{exercise-uniqueness of surprise}
Prove Proposition \ref{prop-uniqueness of surprise} (this is relatively straightforward).
\end{exercise}

Proposition \ref{prop-uniqueness of surprise} says that $H(X)$ does indeed measure the expected amount of surprise evinced by a realization of $X$.

\subsection{Binary entropy} \label{subsec-binary entropy}

We will also use ``$H(\cdot)$'' as notation for a certain function of a single real variable, closely related to entropy.

\begin{defn} \label{defn-binary entropy}
The {\em binary entropy function} is the function $H:[0,1] \rightarrow {\mathbb R}$ given by
$$
H(p) = -p\log p -(1-p)\log (1-p).
$$
Equivalently, $H(p)$ is the entropy of a two-valued (Bernoulli) random variable that takes its two values with probability $p$ and $1-p$.
\end{defn}

$$
\includegraphics[width=3in]{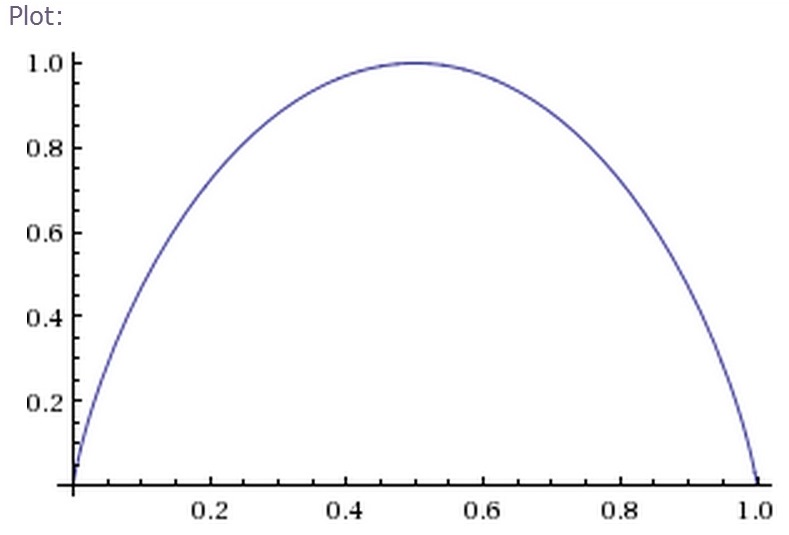}\\
$$

The graph of $H(p)$ is shown above ($x$-axis is $p$). Notice that it has a unique maximum at $p=1/2$ (where it takes the value $1$), rises monotonically from $0$ to $1$ as $p$ goes from $0$ to $1/2$, and falls monotonically back to $0$ as $p$ goes from $1/2$ to $1$. This reflects that idea that there is no randomness in the flip of a two-headed or two-tailed coin ($p=0,1$), and that among biased coins that come up heads with probability $p$, $0 < p < 1$, the fair ($p=1/2$) coin is in some sense the most random.

\subsection{The connection between entropy and counting} \label{subsec-connection to counting}

To see the basic connection between entropy and counting, we need Jensen's inequality.
\begin{thm} \label{thm-jensen}
Let $f:[a,b] \rightarrow {\mathbb R}$ be a continuous, concave function, and let $p_1, \ldots, p_n$ be non-negative reals that sum to $1$. For any $x_1, \ldots, x_n \in [a,b]$,
$$
\sum_{i=1}^n p_i f(x_i) \leq f\left(\sum_{i=1}^n p_ix_i\right).
$$
\end{thm}
Noting that the logarithm function is concave, we have the following corollary, the first basic property of entropy.
\begin{property} (Maximality of the uniform) \label{property-uniform}
For random variable $X$,
$$
H(X) \leq \log |{\rm range}(X)|
$$
where ${\rm range}(X)$ is the set of values that $X$ takes on with positive probability. If $X$ is uniform on its range (taking on each value with probability $1/|{\rm range}(X)|$) then $H(X) = \log |{\rm range}(X)|$.
\end{property}

This property of entropy makes clear why it can be used as an enumeration tool. Suppose ${\mathcal C}$ is some set whose size we want to estimate. If $X$ is a random variable that selects an element from ${\mathcal C}$ uniformly at random, then $|{\mathcal C}| = 2^{H(X)}$, and so anything that can be said about $H(X)$ translates directly into something about $|{\mathcal C}|$.

\subsection{Subadditivity} \label{subsec-subadditivity}

In order to say anything sensible about $H(X)$, and so make entropy a {\em useful}  enumeration tool, we need to derive some further properties. We begin with subadditivity. A vector $(X_1, \ldots, X_n)$ of random variables is itself a random variable, and so we may speak sensibly of $H(X_1, \ldots, X_n)$. Subadditivity relates $H(X_1, \ldots, X_n)$ to $H(X_1)$, $H(X_2)$, etc..
\begin{property} (Subadditivity) \label{property-subadditivity}
For random vector $(X_1, \ldots, X_n)$,
$$
H(X_1, \ldots, X_n) \leq \sum_{i=1}^n H(X_i).
$$
\end{property}
Given the interpretation of entropy as expected surprise, Subadditivity is reasonable: considering the components separately cannot cause less surprise to be evinced than considering them together, since any dependence among the components will only reduce surprise.

We won't prove Subadditivity now, but we will derive it later (Section \ref{subsec-proof of subadditivity}) from a combination of other properties. Subadditivity is all that is needed for our first two applications of entropy, to estimating the sum of binomial coefficients (Section \ref{subsec-binomial}), and (historically the first application of entropy) to obtaining a lower bound for the coin-weighing problem (Section \ref{subsec-Pippinger-first}).

\subsection{Shearer's lemma} \label{subsec-shearer}

Subadditivity was significantly generalized in Chung et al. \cite{ChungFranklGrahamShearer} to what is known as Shearer's lemma. Here and throughout we use $[n]$ for $\{1, \ldots, n\}$.
\begin{lemma} (Shearer's lemma) \label{lemma-Shearer}
Let ${\mathcal F}$ be a family of subsets of $[n]$ (possibly with repeats) with each $i \in [n]$ included in at least $t$ members of ${\mathcal F}$. For random vector $(X_1, \ldots, X_n)$,
$$
H(X_1, \ldots, X_n) \leq \frac{1}{t} \sum_{F \in {\mathcal F}} H(X_F),
$$
where $X_F$ is the vector $(X_i: i \in F)$.
\end{lemma}
To recover Subadditivity from Shearer's lemma, take ${\mathcal F}$ to be the family of singleton subsets of $[n]$. The special case where ${\mathcal F}=\{[n]\setminus i:i \in [n]\}$ is Han's inequality \cite{Han}.

We'll prove Shearer's lemma in Section \ref{subsec-proof of subadditivity}. A nice application to bounding the volume of a body in terms of the volumes of its co-dimension 1 projections is given in Section \ref{subsec-loomiswhitney}, and a more substantial application, to estimating the maximum number of copies of one graph that can appear in another, is given in Section \ref{sec-embedding}.

\subsection{Hiding the entropy in Shearer's lemma} \label{subsec-shearer-counting}

Lemma \ref{lemma-Shearer} does not appear in \cite{ChungFranklGrahamShearer} as we have stated it; the entropy version can be read out of the proof from \cite{ChungFranklGrahamShearer} of the following purely combinatorial version of the lemma. For a set of subsets ${\mathcal A}$ of some ground-set $U$, and a subset $F$ of $U$, the {\em trace} of ${\mathcal A}$ on $F$ is
$$
{\rm trace}_F({\mathcal A}) = \{A \cap F: A \in {\mathcal A}\};
$$
that is, ${\rm trace}_F({\mathcal A})$ is the set of possible intersections of elements of ${\mathcal A}$ with $F$.
\begin{lemma} (Combinatorial Shearer's lemma) \label{lemma-comb-Shearer}
Let ${\mathcal F}$ be a family of subsets of $[n]$ (possibly with repeats) with each $i \in [n]$ included in at least $t$ members of ${\mathcal F}$. Let ${\mathcal A}$ be another set of subsets of $[n]$. Then
$$
|{\mathcal A}| \leq \prod_{F \in {\mathcal F}} |{\rm trace}_F({\mathcal A})|^\frac{1}{t}.
$$
\end{lemma}

\begin{proof}
Let $X$ be an element of ${\mathcal A}$ chosen uniformly at random. View $X$ as the random vector $(X_1, \ldots, X_n)$, with $X_i$ the indicator function of the event $\{i \in X\}$. For each $F \in {\mathcal F}$ we have, using Maximality of the uniform (Property \ref{property-uniform}),
$$
H(X_F) \leq \log |{\rm trace}_F({\mathcal A})|,
$$
as so applying Shearer's lemma with covering family ${\mathcal F}$ we get
$$
H(X) \leq \frac{1}{t}\sum_{F \in {\mathcal F}} \log |{\rm trace}_F({\mathcal A})|.
$$
Using $H(X)=\log |{\mathcal A}|$ (again by Maximality of the uniform) and exponentiating, we get the claimed bound on $|{\mathcal A}|$.
\end{proof}

This proof nicely illustrates the general idea underlying every application of Shearer's lemma: a global problem (understanding $H(X_1, \ldots, X_n)$) is reduced to a collection of local ones (understanding $H(X_F)$ for each $F$), and these local problems can be approached using various properties of entropy.

Some applications of Shearer's lemma in its entropy form could equally well be presented in the purely combinatorial setting of Lemma \ref{lemma-comb-Shearer}; an example is given in Section \ref{subsec-triangles}, where we use Combinatorial Shearer's lemma to estimate the size of the largest family of graphs on $n$ vertices any pair of which have a triangle in common. More complex examples, however, such as those presented in Section \ref{sec-counting colorings}, cannot be framed combinatorially, as they rely on the inherently probabilistic notion of conditioning.

\subsection{Conditional entropy}

Much of the power of entropy comes from being able to understand the relationship between the entropies of dependent random variables. If $E$ is any event, we define the entropy of $X$ given $E$ to be
$$
H(X|E) = \sum_{x} -p(x|E)\log p(x|E),
$$
and for a pair of random variables $X, Y$ we define the entropy of $X$ given $Y$ to be
$$
H(X|Y) = E_Y(H(X|\{Y=y\})) = \sum_y p(y) \sum_{x} ~p(x|y)\log p(x|y).
$$
The basic identity related to conditional entropy is the chain rule, that pins down how the entropy of a random vector can be understood by revealing the components of the vector one-by-one.
\begin{property} (Chain rule) \label{property-chain rule}
For random variables $X$ and $Y$,
$$
H(X,Y) = H(X) + H(Y|X).
$$
More generally,
$$
H(X_1, \ldots, X_n) = \sum_{i=1}^n H(X_i | X_1, \ldots, X_{i-1}).
$$
\end{property}

\begin{proof}
We just prove the first statement, with the second following by induction. For the first,
\begin{eqnarray*}
H(X,Y) - H(X) & = & \sum_{x,y} - p(x,y)\log p(x,y) - \sum_x - p(x)\log p(x) \\
& = & \sum_x p(x) \sum_y -p(y|x)\log p(x)p(y|x)  + \sum_x p(x)\log p(x) \\
& = & \sum_x p(x) \sum_y -p(y|x)\log p(x)p(y|x) + \sum_x p(x) \sum_y p(y|x)\log p(x) \\
& = & \sum_x p(x) \left(\sum_y -p(y|x)\log p(x)p(y|x) + p(y|x)\log p(x) \right) \\
& = & \sum_x p(x) \left(\sum_y -p(y|x)\log p(y|x) \right) \\
& = & H(X|Y).
\end{eqnarray*}
The key point is in the third equality: for each fixed $x$, $\sum_y p(y|x) = 1$.
\end{proof}

Another basic property related to conditional entropy is that increasing conditioning cannot increase entropy. This makes intuitive sense --- the surprise evinced on observing $X$ should not increase if we learn something about it through an observation of $Y$.
\begin{property} (Dropping conditioning) \label{property-adding conditioning}
For random variables $X$ and $Y$,
$$
H(X|Y) \leq H(X).
$$
Also, for random variable $Z$
$$
H(X|Y,Z) \leq H(X|Y).
$$
\end{property}

\begin{proof}
We just prove the first statement, with the proof of the second being almost identical. For the first, we again use the fact that for each fixed $x$, $\sum_y p(y|x) = 1$, which will allow us to apply Jensen's inequality in the inequality below. We also use $p(y)p(x|y)=p(x)p(y|x)$ repeatedly. We have
\begin{eqnarray*}
H(X|Y) & = & \sum_{y} p(y) \sum_{x} - p(x|y)\log p(x|y) \\
& = & \sum_x p(x) \sum_y -p(y|x)\log p(x|y) \\
& \leq & \sum_x p(x) \log \left(\sum_y \frac{p(y|x)}{p(x|y)}\right) \\
& = & \sum_x p(x) \log \left(\sum_y \frac{p(y)}{p(x)}\right) \\
& = & \sum_x - p(x) \log p(x) \\
& = & H(X).
\end{eqnarray*}
\end{proof}

\subsection{Proofs of Subadditivity (Property \ref{property-subadditivity}) and Shearer's lemma (Lemma \ref{lemma-Shearer})} \label{subsec-proof of subadditivity}

The subadditivity of entropy follows immediately from a combination of the Chain rule (Property \ref{property-chain rule}) and Dropping conditioning (Property \ref{property-adding conditioning}).

The original proof of Shearer's lemma from \cite{ChungFranklGrahamShearer} involved an intricate and clever induction. Radhakrishnan and Llewellyn (reported in \cite{Radhakrishnan2}) gave the following lovely proof using the Chain rule and Dropping conditioning.

Write $F \in {\mathcal F}$ as $F =\{i_1, \ldots, i_k\}$ with $i_1 < i_2 < \ldots < i_k$. We have
\begin{eqnarray*}
H(X_F) & = & H(X_{i_1}, \ldots, X_{i_k}) \\
& = & \sum_{j=1}^k H(X_{i_j} | (X_{i_\ell}: \ell < j)) \\
& \geq & \sum_{j=1}^k H(X_{i_j} | X_1, \ldots, X_{i_j-1}).
\end{eqnarray*}
The inequality here is an application of Dropping conditioning.
If we sum this last expression over all $F \in {\mathcal F}$, then for each $i \in [n]$ the term $H(X_i|X_1, \ldots, X_{i-1})$ appears at least $t$ times and so
\begin{eqnarray*}
\sum_{F \in {\mathcal F}} H(X_F) & \geq & t \sum_{i=1}^n H(X_i|X_1, \ldots, X_{i-1}) \\
& = & tH(X),
\end{eqnarray*}
the equality using the Chain rule. Dividing through by $t$ we obtain Shearer's lemma.

\subsection{Conditional versions of the basic properties}

Conditional versions of each of Maximality of the uniform, the Chain rule, Subadditivity, and Shearer's lemma are easily proven, and we merely state the results here.

\begin{property} (Conditional maximality of the uniform) \label{property-conditonal-uniform}
For random variable $X$ and event $E$,
$$
H(X|E) \leq \log |{\rm range}(X|E)|
$$
where ${\rm range}(X|E)$ is the set of values that $X$ takes on with positive probability, given that $E$ has occurred.
\end{property}

\begin{property} (Conditional chain rule) \label{property-conditional-chain rule}
For random variables $X$, $Y$ and $Z$,
$$
H(X,Y|Z) = H(X|Z) + H(Y|X,Z).
$$
More generally,
$$
H(X_1, \ldots, X_n|Z) = \sum_{i=1}^n H(X_i | X_1, \ldots, X_{i-1}, Z).
$$
\end{property}

\begin{property} (Conditional subadditivity) \label{property-conditional-subadditivity}
For random vector $(X_1, \ldots, X_n)$, and random variable $Z$,
$$
H(X_1, \ldots, X_n|Z) \leq \sum_{i=1}^n H(X_i|Z).
$$
\end{property}

\begin{lemma} (First conditional Shearer's lemma) \label{lemma-conditional Shearer 1}
Let ${\mathcal F}$ be a family of subsets of $[n]$ (possibly with repeats) with each $i \in [n]$ included in at least $t$ members of ${\mathcal F}$. For random vector $(X_1, \ldots, X_n)$ and random variable $Z$,
$$
H(X_1, \ldots, X_n|Z) \leq \frac{1}{t} \sum_{F \in {\mathcal F}} H(X_F|Z).
$$
\end{lemma}

A rather more powerful and useful conditional version of Shearer's lemma, that may be proved exactly as we proved Lemma \ref{lemma-Shearer}, was given by Kahn \cite{Kahn2}.
\begin{lemma} (Second conditional Shearer's lemma) \label{lemma-conditional Shearer 2}
Let ${\mathcal F}$ be a family of subsets of $[n]$ (possibly with repeats) with each $i \in [n]$ is included in at least $t$ members of ${\mathcal F}$. Let $\prec$ be a partial order on $[n]$, and for $F \in {\mathcal F}$ say that $i \prec F$ if $i \prec x$ for each $x \in F$.
For random vector $(X_1, \ldots, X_n)$,
$$
H(X_1, \ldots, X_n) \leq \frac{1}{t} \sum_{F \in {\mathcal F}} H(X_F|\{X_i:i \prec F\}).
$$
\end{lemma}

\begin{exercise} \label{exercise -- prove properties}
Give proofs of all the properties and lemmas from this section.
\end{exercise}

\section{Four quick applications} \label{sec-quick applications}

Here we give four fairly quick applications of the entropy method in combinatorial enumeration.

\subsection{Sums of binomial coefficients} \label{subsec-binomial}

There is clearly a connection between entropy and the binomial coefficients; for example, Stirling's approximation to $n!$ ($n! \sim (n/e)^n \sqrt{2\pi n}$ as $n \rightarrow \infty$) gives
\begin{equation}
\binom{n}{\alpha n}  \sim \frac{2^{H(\alpha)n}}{\sqrt{2\pi n\alpha(1-\alpha)}} \label{eq-n choose alpha n}
\end{equation}
for any fixed $0 < \alpha < 1$. Here is a nice bound on the sum of all the binomial coefficients up to $\alpha n$, that in light of (\ref{eq-n choose alpha n}) is relatively tight, and whose proof nicely illustrates the use of entropy.
\begin{thm} \label{thm-sums of binomial coefficients}
Fix $\alpha \leq 1/2$. For all $n$,
$$
\sum_{i \leq \alpha n} \binom{n}{i} \leq 2^{H(\alpha)n}.
$$
\end{thm}

\begin{proof}
Let ${\mathcal C}$ be the set of all subsets of $[n]$ of size at most $\alpha n$; note that $|{\mathcal C}|=\sum_{i \leq \alpha n} \binom{n}{i}$. Let $X$ be a uniformly chosen member of ${\mathcal C}$; by Maximality of the uniform, it is enough to show $H(X) \leq H(\alpha)n$.

View $X$ as the random vector $(X_1, \ldots, X_n)$, where $X_i$ is the indicator function of the event $\{i \in X\}$. By Subadditivity and symmetry,
$$
H(X) \leq H(X_1) + \ldots + H(X_n) = nH(X_1).
$$
So now it is enough to show $H(X_1) \leq H(\alpha)$. To see that this is true, note that $H(X_1) = H(p)$, where $p=\Pr(1 \in X)$. We have $p \leq \alpha$ (conditioned on $X$ having size $\alpha n$, $\Pr(i \in X)$ is exactly $\alpha$, and conditioned on $X$ having any other size it is strictly less than $\alpha$), and so, since $\alpha \leq 1/2$, $H(p) \leq H(\alpha)$.
\end{proof}

Theorem \ref{thm-sums of binomial coefficients} can be used to quickly obtain the following concentration inequality for the balanced ($p=1/2$) binomial distribution, a weak form of the Chernoff bound.
\begin{exercise} \label{exercise-chernof}
Let $X$ be a binomial random variable with parameters $n$ and $1/2$. Show that for every $c \geq 0$,
$$
\Pr(|X - n/2| \geq c\sigma) \leq 2^{1-c^2/2},
$$
where $\sigma = \sqrt{n}/2$ is the standard deviation of $X$.
\end{exercise}

\subsection{The coin-weighing problem} \label{subsec-Pippinger-first}

Suppose we are given $n$ coins, some of which are pure and weigh $a$ grams, and some of which are counterfeit and weight $b$ grams, with $b<a$. We are given access to an accurate scale (not a balance), and wish to determine which are the counterfeit coins using as few weighings as possible, with a sequence of weighings announced in advance. How many weighings are needed to isolate the counterfeit coins? (A very specific version of this problem is due to Shapiro \cite{Shapiro}.)

When a set of coins is weighed, the information obtained is the number of counterfeit coins among that set. Suppose that we index the coins by elements of $[n]$. If the sequence of subsets of coins that we weigh is $D_1, \ldots, D_\ell$, then the set ${\mathcal D}=\{D_1, \ldots, D_\ell\}$ must form what is called a {\em distinguishing family} for $[n]$ --- it must be such that for every $A, B \subseteq [n]$ with $A \neq B$, there is a $D_i \in {\mathcal D}$ with $|A \cap D_i| \neq |B \cap D_i|$ --- for if not, and the $D_i$'s fail to distinguish a particular pair $A, B$, then our weighings would not be able distinguish between $A$ or $B$ being the set of counterfeit coins. On the other hand, if the $D_i$ do form a distinguishing family, then they also form a good collection of weighings --- if $A$ is the collection of counterfeit coins, then on observing the vector $(|A \cap D_i|: i=1, \ldots, \ell)$ we can determine $A$, since $A$ is the unique subset of $[n]$ that gives rise to that particular vector of intersections.

It follows that determining the minimum number of weighings required is equivalent to the combinatorial question of determining $f(n)$, the minimum size of a distinguishing family for $[n]$. Cantor and Mills \cite{CantorMills} and Lindstr\"om \cite{Lindstrom} independently established the upper bound
$$
f(n) \leq \frac{2n}{\log n}\left(1+ O\left(\frac{\log \log n}{\log n}\right)\right)
$$
while Erd\H{o}s and R\'enyi \cite{ErdosRenyi} and (independently) Moser \cite{Moser} obtained the lower bound
\begin{equation} \label{eq-coinlb}
f(n) \geq \frac{2n}{\log n}\left(1+ \Omega\left(\frac{1}{\log n}\right)\right).
\end{equation}
(See the note at the end of \cite{ErdosRenyi} for the rather involved history of these bounds).
Here we give a short entropy proof of (\ref{eq-coinlb}). A proof via information theory of a result a factor of 2 weaker was described (informally) by Erd\H{o}s and R\'enyi \cite{ErdosRenyi}; to the best of our knowledge this is the first application of ideas from information theory to a combinatorial problem. Pippinger \cite{Pippinger} recovered the factor of 2 via a more careful entropy argument.

Let $X$ be a uniformly chosen subset of $[n]$, so that (by Maximality of the uniform) $H(X) =n$. By the discussion earlier, observing $X$ is equivalent to observing the vector $(|X \cap D_i|:i = 1, \ldots, \ell)$ (both random variables have the same distribution), and so
$$
H(X) = H((|X \cap D_i|:i = 1, \ldots, \ell)) \leq \sum_{i=1}^\ell H(|X\cap D_i|),
$$
the inequality by Subadditivity. Since $|X\cap D_i|$ can take on at most $n+1$ values, we have (again by Maximality of the uniform) $H(|X\cap D_i|) \leq \log(n+1)$. Putting all this together we obtain (as Erd\H{o}s and R\'enyi did)
$$
n=H(X) \leq \ell \log(n+1)
$$
or $\ell \geq n/\log(n+1)$, which falls short of (\ref{eq-coinlb}) by a factor of $2$.

To gain back this factor of $2$, we need to be more careful in estimating $H(|X\cap D_i|)$. Observe that $|X \cap D_i|$ is a binomial random variable with parameters $d_i$ and $1/2$, where $d_i=|D_i|$, and so has entropy
$$
\sum_{j=0}^{d_i} \binom{d_i}{j} 2^{-j} \log \left(\frac{2^{d_i}}{\binom{d_i}{j}}\right).
$$
If we can show that this is at most $(1/2)\log d_i + C$ (where $C$ is some absolute constant), then the argument above gives
$$
n \leq \ell \left(\frac{1}{2} \log n + O(1)\right),
$$
which implies (\ref{eq-coinlb}). We leave the estimation of the binomial random variable's entropy as an exercise; the intuition is that the vast majority of the mass of the binomial is within 10 (say) standard deviations of the mean (a consequence, for example, of Exercise \ref{exercise-chernof}, but Tchebychev's inequality would work fine here), and so only $\sqrt{d_i}$ of the possible values that the binomial takes on contribute significantly to its entropy.
\begin{exercise} \label{exercise-binentest}
Show that there's a constant $C>0$ such that for all $m$,
$$
\sum_{j=0}^m \binom{m}{j} 2^{-j} \log \left(\frac{2^m}{\binom{m}{j}}\right) \leq \frac{\log m}{2} + C.
$$
\end{exercise}


\subsection{The Loomis-Whitney theorem} \label{subsec-loomiswhitney}

How large can a measurable body in ${\mathbb R}^n$ be, in terms of the volumes of its $(n-1)$-dimensional projections? The following theorem of Loomis and Whitney \cite{LoomisWhitney} gives a tight bound. For a measurable body $B$ in ${\mathbb R}^n$, and for each $j \in [n]$, let $B_j$ be the projection of $B$ onto the hyperplane $x_j=0$; that is, $B_j$ is the set of all $(x_1, \ldots, x_{j-1}, x_{j+1}, \ldots, x_n)$ such that there is some $x_j \in {\mathbb R}$ with $(x_1, \ldots, x_{j-1}, x_j, x_{j+1}, \ldots, x_n) \in B$.
\begin{thm} \label{thm-loomis-whitney}
Let $B$ be a measurable body in ${\mathbb R}^n$. Writing $|\cdot|$ for volume,
$$
|B| \leq \prod_{j=1}^n |B_j|^{1/(n-1)}.
$$
This bound is tight, for example when $B$ is a cube.
\end{thm}

\begin{proof}
We prove the result in the case when $B$ is a union of axis-parallel cubes with side-lengths $1$ centered at points with integer coordinates (and we identify a cube with the coordinates of its center); the general result follows from standard scaling and limiting arguments.

Let $X$ be a uniformly selected cube from $B$; we write $X$ as $(X_1, \ldots, X_n)$, where $X_i$ is the $i$th coordinate of the cube. We upper bound $H(X)$ by applying Shearer's lemma (Lemma \ref{lemma-Shearer}) with ${\mathcal F} = \{F_1, \ldots, F_n\}$, where $F_j=[n] \setminus j$. For this choice of ${\mathcal F}$ we have $t=n-1$. The support of $X_{F_j}$ (i.e., the set of values taken by $X_{F_j}$ with positive probability) is exactly (the set of centers of the $(d-1)$-dimensional cubes comprising) $B_j$. So, using Maximality of the uniform twice, we have
\begin{eqnarray*}
\log |B| & = & H(X) \\
& \leq & \frac{1}{n-1}\sum_{j=1}^n H(X_{F_j}) \\
& \leq & \frac{1}{n-1}\sum_{j=1}^n \log |B_j|,
\end{eqnarray*}
from  which the theorem follows.
\end{proof}

\subsection{Intersecting families} \label{subsec-triangles}

Let ${\mathcal G}$ be a family of graphs on vertex set $[n]$, with the property that for each $G_1, G_2 \in {\mathcal G}$, $G_1 \cap G_2$ contains a triangle (i.e, there are three vertices $i, j, k$ such that each of $ij$, $ik$, $jk$ is in the edge set of both $G_1$ and $G_2$). At most how large can ${\mathcal G}$ be? This question was first raised by Simonovits and S\'os in 1976.

Certainly $|{\mathcal G}|$ can be as large as $2^{\binom{n}{2}-3}$: consider the family ${\mathcal G}$ of all graphs that include a particular triangle. In the other direction, it can't be larger than $2^{\binom{n}{2}-1}$, by virtue of the well-known result that a family of distinct sets on ground set of size $m$, with the property that any two members of the family have non-empty intersection, can have cardinality at most $2^{m-1}$ (the edge sets of elements of ${\mathcal G}$ certainly form such a family, with $m=\binom{n}{2}$). In \cite{ChungFranklGrahamShearer} Shearer's lemma is used to improve this easy upper bound.
\begin{thm} \label{thm-triangle intersection}
With ${\mathcal G}$ as above, $|{\mathcal G}| \leq 2^{\binom{n}{2}-2}$.
\end{thm}

\begin{proof}
Identify each graph $G \in {\mathcal G}$ with its edge set, so that ${\mathcal G}$ is now a set of subsets of a ground-set $U$ of size $\binom{n}{2}$. For each unordered equipartition $A \cup B = [n]$ (satisfying $\left||A|-|B|\right| \leq 1$), let $U(A,B)$ be the subset of $U$ consisting of all those edges that lie entirely inside $A$ or entirely inside $B$. We will apply Combinatorial Shearer's lemma with ${\mathcal F} = \{U(A,B)\}$.

Let $m=|U(A,B)|$ (this is independent of the particular choice of equipartition). Note that
$$
m = \left\{
\begin{array}{ll}
2\binom{n/2}{2} & \mbox{if $n$ is even}\\
\binom{\lfloor n/2 \rfloor}{2} + \binom{\lceil n/2 \rceil}{2} & \mbox{if $n$ is odd;}\\
\end{array}
\right.
$$
in either case, $m \leq \frac{1}{2}\binom{n}{2}$. Note also that by a simple double-counting argument we have
\begin{equation} \label{eq-doublecounting}
m|{\mathcal F}| = \binom{n}{2} t
\end{equation}
where $t$ is the number of elements of ${\mathcal F}$ in which each element of $U$ occurs.

Observe that ${\rm trace}_{U(A,B)}({\mathcal G})$ forms an intersecting family of subsets of $U(A,B)$; indeed, for any $G, G' \in {\mathcal G}$, $G \cap G'$ has a triangle $T$, and since the complement of $U(A,B)$ (in $U$) is triangle-free (viewed as a graph on $[n]$), at least one of the edges of $T$ must meet $U(A,B)$. So,
$$
|{\rm trace}_{U(A,B)} ({\mathcal G})| \leq 2^{m - 1}.
$$
By Lemma \ref{lemma-comb-Shearer},
\begin{eqnarray*}
|{\mathcal G}| & \leq & \left(2^{m - 1}\right)^\frac{|{\mathcal F}|}{t} \\
& = & 2^{\binom{n}{2}\left(1-\frac{1}{m}\right)} \\
& \leq & 2^{\binom{n}{2}-2},
\end{eqnarray*}
as claimed (the equality here uses (\ref{eq-doublecounting})).
\end{proof}

Recently Ellis, Filmus and Friedgut \cite{EllisFilmusFriedgut} used discrete Fourier analysis to obtain the sharp bound $|{\mathcal G}| \leq 2^{\binom{n}{2}-3}$ that had been conjectured by Simonovits and S\'os.

\section{Embedding copies of one graph in another} \label{sec-embedding}

We now move on to our first more substantial application of entropy to combinatorial enumeration; the problem of maximizing the number of copies of a graph that can be embedded in a graph on a fixed number of edges.

\subsection{Introduction to the problem}  \label{subsec-embeddingintro}

At most how many copies of a fixed graph $H$ can there be in a graph with $\ell$ edges? More formally, define an {\em embedding} of a graph $H$ into a graph $G$ as an injective function $f$ from $V(H)$ to $V(G)$ with the property that $f(x)f(y) \in E(G)$ whenever $xy \in E(H)$. Let ${\rm embed}(H,G)$ be the number of embeddings of $H$ into $G$, and let ${\rm embed}(H,\ell)$ be the maximum of ${\rm embed}(H,G)$ as $G$ varies over all graphs on $\ell$ edges. The question we are asking is: what is the value of ${\rm embed}(H,\ell)$ for each $H$ and $\ell$?

Consider for example $H=K_3$, the triangle. Fix $G$ with $\ell$ edges. Suppose that $x \in V(H)$ is mapped to $v \in V(G)$. At most how many ways can this partial embedding be completed? Certainly no more that $2\ell$ ways (the remaining two vertices of $H$ must be mapped, in an ordered way, to one of the $\ell$ edges of $G$); but also, no more than $d_v(d_v-1) \leq d_v^2$ ways, where $d_v$ is the degree of $v$ (the remaining two vertices of $H$ must be mapped, in an ordered way, to neighbors of $v$). Since $\min\{d_v^2, 2\ell\} \leq d_v\sqrt{2\ell}$, a simple union bound gives
$$
{\rm embed}(H,G) \leq \sum_{v \in V(G)} d_v\sqrt{2\ell} = 2\sqrt{2} \ell^{3/2},
$$
and so ${\rm embed}(H,\ell) \leq 2\sqrt{2} \ell^{3/2}$. On the other hand, this is the right order of magnitude, since the complete graph of $\sqrt{2\ell}$ vertices admits $\sqrt{2\ell}(\sqrt{2\ell}-1)(\sqrt{2\ell}-2) \approx 2\sqrt{2}\ell^{3/2}$ embeddings of $K_3$, and has around $\ell$ edges.

The following theorem was first proved by Alon \cite{Alon2}. In Section \ref{subsec-embedproofs} we give a proof based on Shearer's lemma due to Friedgut and Kahn \cite{FriedgutKahn}. The definition of $\rho^\star$ is given in Section \ref{subsec-fractional}.
\begin{thm} \label{thm-embeddingub}
For all graphs $H$ there is a constant $c_1>0$ such that for all $\ell$,
$$
{\rm embed}(H,\ell) \leq c_1 \ell^{\rho^\star(H)}
$$
where $\rho^\star(H)$ is the fractional cover number of $H$.
\end{thm}
There is a lower bound that matches the upper bound up to a constant.
\begin{thm} \label{thm-embeddinglb}
For all graphs $H$ there is a constant $c_2>0$ such that for all $\ell$,
$$
{\rm embed}(H,\ell) \geq c_2 \ell^{\rho^\star(H)}.
$$
\end{thm}

\subsection{Background on fractional covering and independent sets} \label{subsec-fractional}

A vertex cover of a graph $H$ is a set of edges with each vertex included in at least one edge in the set, and the vertex cover number $\rho(H)$ is defined to be the minimum number of edges in a vertex cover. Equivalently, we may define a cover function to be a $\varphi:E(H) \rightarrow \{0,1\}$ satisfying
\begin{equation} \label{eq-coverconstaints}
\sum_{e \in E(H)~:~v \in e} \varphi(e) \geq 1
\end{equation}
for each $v \in V(H)$, and then define $\rho(H)$ to be the minimum of $\sum_{e \in E(H)} \varphi(e)$ over all cover functions $\varphi$.

This second formulation allows us to define a {\em fractional} version of the cover number, by relaxing the condition that $\varphi(e)$ must be an integer. Define a {\em fractional} cover function to be a $\varphi:E(H) \rightarrow [0,1]$ satisfying (\ref{eq-coverconstaints})
for each $v \in V(H)$, and then define $\rho^\star(H)$ to be the minimum of $\sum_{e \in E(H)} \varphi(e)$ over all fractional cover functions $\varphi$; note that $\rho^\star(H) \leq \rho(H)$.

An independent set in $H$ is a set of vertices with each edge touching at most one vertex in the set, and the independence number $\alpha(H)$ is defined to be the maximum number of vertices in an independent set. Equivalently, define an independence function to be a $\psi:V(H) \rightarrow \{0,1\}$ satisfying
\begin{equation} \label{eq-indconstaints}
\sum_{v \in V(H)~:~v \in e} \psi(v) \leq 1
\end{equation}
for each $e \in E(H)$, and then define $\alpha(H)$ to be the maximum of $\sum_{v \in V(H)} \psi(v)$ over all independence functions $\psi$. Define a {\em fractional} independence function to be a $\psi:V(H) \rightarrow [0,1]$ satisfying (\ref{eq-indconstaints})
for each $e \in E(H)$, and then define $\alpha^\star(H)$ to be the maximum of $\sum_{v \in V(H)} \psi(v)$ over all fractional independence functions $\psi$; note that $\alpha(H) \leq \alpha^\star(H)$.

We always have $\alpha(H) \leq \rho(H)$ (a vertex cover needs to use a different edge for each vertex of an independent set), and usually $\alpha(H) < \rho(H)$ (as for example when $H=K_3$). The gap between these two parameters closes, however, when we pass to the fractional variants. By the fundamental theorem of linear programming duality we have
\begin{equation} \label{eq-duality}
\alpha^\star(H) = \rho^\star(H)
\end{equation}
for every $H$.

\subsection{Proofs of Theorem \ref{thm-embeddingub} and \ref{thm-embeddinglb}} \label{subsec-embedproofs}

We begin with Theorem \ref{thm-embeddingub}.
Let the vertices of $G$ be $\{v_1, \ldots, v_{|V(H)|}\}$. Let $G$ be a fixed graph on $\ell$ edges and let $X$ be a uniformly chosen embedding of $H$ into $G$. Encode $X$ as the vector $(X_1, \ldots, X_{|V(H)|})$, where $X_i$ is the vertex of $G$ that $i$ is mapped to by $X$. By Maximality of the uniform, $H(X) = \log ({\rm embed}(H,G))$, so if we can show $H(X) \leq \rho^\star(H) \log (c \ell)$ for some constant $c=c(H)$ then we are done.

Let $\varphi^\star:E(H) \rightarrow [0,1]$ be an optimal fractional vertex cover of $H$, that is, one satisfying $\sum_{e \in E(H)} \varphi(e) = \rho^\star(H)$. We may assume that $\varphi(e)$ is rational for all $e \in E(H)$, and we may choose an integer $C$ such that $C\varphi(e)$ is an integer for each such $e$.

We will apply Shearer's lemma with ${\mathcal F}$ consisting of $C\varphi^\star(e)$ copies of the pair $\{u,v\}$, where $e=uv$, for each $e \in E(H)$. Each $v \in V(H)$ appears in at least
$$
\sum_{e \in E(H):v \in e} C\varphi^\star(e) \geq C
$$
members of ${\mathcal F}$ (the inequality using (\ref{eq-coverconstaints})), so by Shearer's lemma
\begin{eqnarray*}
H(X) & \leq & \frac{1}{C} \sum_{e=uv \in E(H)} C\varphi^\star(e)H(X_u,X_v) \\
& \leq & \sum_{e \in E(H)} \varphi^\star(e)\log (2\ell) \\
& = & \rho^\star(H) \log (2 \ell),
\end{eqnarray*}
as required. The second inequality above uses Maximality of the uniform ($X_u$ and $X_v$ must be a pair of adjacent vertices, and there are $2\ell$ such pairs in a graph on $\ell$ edges), and the equality uses the fact that $\varphi^\star$ is an optimal fractional vertex cover.

For the proof of Theorem \ref{thm-embeddinglb}, by (\ref{eq-duality}) it is enough to exhibit, for each $\ell$, a single graph $G_\ell$ on at most $\ell$ edges for which ${\rm embed}(H,G_\ell) \geq c_2 \ell^{\alpha^\star(H)}$, where the constant $c_2>0$ is independent of $\ell$. Let $\psi^\star:V(H)\rightarrow [0,1]$ be an optimal fractional independence function (one satisfying $\sum_{v \in V(H)} \psi(v) = \alpha^\star(H)$). Create a graph $H^\star$ on vertex set $\cup_{v \in V(H)} V(v)$, where the $V(v)$'s are disjoint sets with $|V(v)| = (\ell/|E(H)|)^{\psi^\star(v)}$ for each $v \in V(H)$, and with an edge between two vertices exactly when one is in $V(v)$ and the other is in $V(w)$ for some $vw \in E(H)$ (note that $|V(v)|$ as defined may not be an integer, but this makes no essential difference to the argument, and dealing with this issue formally only obscures the proof with excessive notation).

The number of edges in $H^\star$ is
$$
\sum_{e = vw \in E(H)} \left(\frac{\ell}{|E(H)|}\right)^{\psi^\star(v) + \psi^\star(w)} \leq \sum_{e = vw \in E(H)} \frac{\ell}{|E(H)|} \leq \ell,
$$
the first inequality using (\ref{eq-indconstaints}). Any function $f:V(H) \rightarrow V(H^\star)$ satisfying $f(v) \in V(v)$ for each $v \in V(H)$ is an embedding of $H$ into $H^\star$, and so
$$
{\rm embed}(H,H^\star) \geq \left(\frac{\ell}{|E(H)|}\right)^{\sum_{v \in V(H)} \psi^\star(v)} = \left(\frac{\ell}{|E(H)|}\right)^{\alpha^\star(H)},
$$
the equality using the fact that $\psi^\star$ is an optimal fractional independent set.

\section{Br\'egman's theorem (the Minc conjecture)} \label{sec-bregman}

Here we present Radhakrishnan's beautiful entropy proof \cite{Radhakrishnan} of Br\'egman's theorem on the maximum permanent of a 0-1 matrix with given row sums.

\subsection{Introduction to the problem} \label{subsec-intro to Bregman}

The {\em permanent} of an $n$ by $n$ matrix $A=(a_{ij})$ is
$$
{\rm perm}(A) = \sum_{\sigma \in S_n} \prod_{i=1}^n a_{i\sigma(i)}
$$
where $S_n$ is the set of permutations of $[n]$. This seems superficially quite similar to the determinant, which differs only by the addition of a factor of $(-1)^{{\rm sgn}(\sigma)}$ in front of the product. This small difference makes all the difference, however: problems involving the determinant are generally quite tractable algorithmically (because Gaussian elimination can be performed efficiently), but permanent problems seems to be quite intractable (in particular, by a Theorem of Valiant \cite{Valiant} the computation of the permanent of a general $n$ by $n$ matrix is $\# P$-hard).

The permanent of a $0$-$1$ matrix has a nice interpretation in terms of perfect matchings (1-regular spanning subgraphs) in a graph. There is a natural one-to-one correspondence between $0$-$1$ $n$ by $n$ matrices and bipartite graphs on fixed color classes each of size $n$: given $A=(a_{ij})$ we construct a bipartite graph $G=G(A)$ on color classes ${\mathcal E} =\{v_1, \ldots, v_n\}$ and ${\mathcal O}=\{w_1, \ldots, w_n\}$ by putting $v_iw_j \in E$ if and only if $a_{ij}=1$. Each $\sigma \in S_n$ that contributes $1$ to ${\rm perm}(A)$ gives rise to the perfect matching ($1$-regular spanning subgraph) $\{v_iw_{\sigma(i)}:i \in [n]\}$ in $G$, and this correspondence is bijective; all other $\sigma \in S_n$ contribute $0$ to ${\rm perm}(A)$. In other words,
$$
{\rm perm}(A) = |{\mathcal M}_{\rm perf}(G)|
$$
where ${\mathcal M}_{\rm perf}(G)$ is the set of perfect matchings of $G$.

In 1963 Minc formulated a natural conjecture concerning the permanent of an $n$ by $n$ $0$-$1$ matrix with all row sums fixed. Ten years later Br\'egman \cite{Bregman} gave the first proof, and the result is now known as Br\'egman's theorem.
\begin{thm} \label{thm-MincBregman} (Br\'egman's theorem)
Let $n$ non-negative integers $d_1, \ldots, d_n$ be given. Let $A=(a_{ij})$ be an $n$ by $n$ matrix with all entries in $\{0,1\}$ and with $\sum_{j=1}^n a_{ij} = d_i$ for each $i=1, \ldots, n$ (that is, with the sum of the row $i$ entries of $A$ being $d_i$, for each $i$). Then
$$
{\rm perm}(A) \leq \prod_{i=1}^n \left(d_i!\right)^\frac{1}{d_i}.
$$
Equivalently, let $G$ be a bipartite graph on color classes ${\mathcal E}=\{v_1, \ldots, v_n\}$, ${\mathcal O}=\{w_1, \ldots, w_n\}$, with each $v_i \in {\mathcal E}$ having degree $d_i$. Then
$$
|{\mathcal M}_{\rm perf}(G)| \leq \prod_{i=1}^n \left(d_i!\right)^\frac{1}{d_i}.
$$
\end{thm}
Notice that the bound is tight: for example, for each fixed $d$ and $n$ with $d|n$, it is achieved by the matrix consisting of $n/d$ blocks down the diagonal with each block being a $d$ by $d$ matrix of all $1$'s, and with zeros everywhere else (or equivalently, by the graph made up of the disjoint union of $n/d$ copies of $K_{d,d}$, the complete bipartite graph with $d$ vertices in each classes).

A short proof of Br\'egman's theorem was given by Schrijver \cite{Schrijver}, and a probabilistic reinterpretation of Schrijver's proof was given by Alon and Spencer \cite{AlonSpencer}. A beautiful proof using subadditivity of entropy was given by Radhakrishnan \cite{Radhakrishnan}, and we present this in Section \ref{subsec-rads proof}. Many interesting open questions remain in this area; we present some of these in Section \ref{subsec-matching open}.

\medskip

Br\'egman's theorem concerns perfect matchings in a bipartite graph. A natural question to ask is: what happens in a general (not necessarily bipartite) graph? Kahn and Lov\'asz answered this question.
\begin{thm} \label{thm-KL} (Kahn-Lov\'asz theorem)
Let $G$ be a graph on $2n$ vertices $v_1, \ldots, v_{2n}$ with each $v_i$ having degree $d_i$. Then
$$
|{\mathcal M}_{\rm perf}(G)| \leq \prod_{i=1}^{2n} \left(d_i!\right)^\frac{1}{2d_i}.
$$
\end{thm}
Notice that this result is also tight: for example, for each fixed $d$ and $n$ with $d|n$, it is achieved by the graph made up of the disjoint union of $n/d$ copies of $K_{d,d}$. Note also that there is no permanent version of this result.

Kahn and Lov\'asz did not publish their proof. Since they first discovered the theorem, it has been rediscovered/reproved a number of times: by Alon and Friedland \cite{AlonFriedland}, Cutler and Radcliffe \cite{CutlerRadcliffe}, Egorychev \cite{Egorychev} and Friedland \cite{Friedland}. Alon and Friedland's is a ``book'' proof, observing that the theorem is an easy consequence of Br\'egman's theorem. We present the details in Section \ref{subsec-AlonFriedland}.

\subsection{Radhakrishnan's proof of Br\'egman's theorem} \label{subsec-rads proof}

A perfect matching $M$ in $G$ may be encoded as a bijection $f:[n] \rightarrow [n]$ via $f(i)=j$ if and only if $v_iw_j \in M$. This is how we will view matchings from now on. Let $X$ be a random variable which represents the uniform selection of a matching $f$ from ${\mathcal M}_{\rm perf}(G)$, the set of all perfect matchings in $G$. By Maximality of the uniform, $H(X) = \log |{\mathcal M}_{\rm perf}(G)|$,
and so our goal is to prove
\begin{equation} \label{inq-target}
H(X) \leq \sum_{k=1}^n \frac{\log d_k!}{d_k}.
\end{equation}
We view $X$ as the random vector $(f(1), \ldots, f(n))$. By Subadditivity, $H(X) \leq \sum_{k=1}^n H(f(k))$.
Since there are at most $d_i$ possibilities for the value of $f(k)$, we have $H(f(k)) \leq \log d_k$
for all $k$, and so $H(X) \leq \sum_{k=1}^n \log d_k$.
This falls somewhat short of (\ref{inq-target}), since $(\log d_k!)/d_k \approx \log (d_k/e)$ by Stirling's approximation to the factorial function.

We might try to improve things by using the sharper Chain rule in place of Subadditivity:
$$
H(X) = \sum_{k=1}^n H(f(k)|f(1),\ldots,f(k-1)).
$$
Now instead of naively saying that there are $d_k$ possibilities for $f(k)$ for each $k$, we have a chance to take into account the fact that when it comes time to reveal $f(k)$, some of $v_k$'s neighbors may have already been used (as a match for $v_j$ for some $j < k$), and so there may be a reduced range of choices for $f(k)$; for example, we can say definitively that $H(f(n)|f(1),\ldots,f(n-1))=0$.

The problem with this approach is that in general we have no way of knowing (or controlling) how many neighbors of $k$ have been used at the moment when $f(k)$ is revealed. Radhakrishnan's idea to deal with this problem is to choose a {\em random} order in which to examine the vertices of ${\mathcal E}$ (rather than the deterministic order $v_1, \ldots, v_n$). There is a good chance that with a random order, we can say something precise about the average or expected number of neighbors of $k$ that have been used at the moment when $f(k)$ is revealed, and thereby put a better upper bound on the $H(f(k))$ term.

So, let $\tau = \tau_1\ldots \tau_n$ be a permutation of $[n]$ (which we will think of as acting on ${\mathcal E}$ in the natural way). We have
$$
H(X) = \sum_{k=1}^n H(f(\tau_k)|f(\tau_1),\ldots,f(\tau_{k-1})).
$$
It will prove convenient to re-write this as
\begin{equation} \label{eq-tauchainrule}
H(X) = \sum_{k=1}^n H(f(k)|(f(\tau_\ell):\ell < \tau^{-1}_k)),
\end{equation}
where $\tau^{-1}_k$ is the element of $[n]$ that $\tau$ maps to $k$. Averaging (\ref{eq-tauchainrule}) over all $\tau$ (and changing order of summation) we obtain
\begin{equation} \label{eq-averagechainrule}
H(X) = \sum_{k=1}^n \frac{1}{n!} \sum_{\tau \in S_n} H(f(\tau_k)|(f(\tau_\ell):\ell < \tau^{-1}_k)).
\end{equation}
From here on we fix $k$ and examine the summand in (\ref{eq-averagechainrule}) corresponding to $k$.

For fixed $\tau \in S_n$ and $f \in {\mathcal M}_{\rm perf}(G)$, let $N_k(\tau,f)$ denote the number of $i \in [n]$ such that $v_kw_i \in E(G)$, and $i \not \in \{f(\tau_1),\ldots,f(\tau_{k-1})\}$ (in other words, $N_k(\tau,f)$ is the number of possibilities that remain for $f(k)$ when $f(\tau_1),\ldots,f(\tau_{k-1})$ have all been revealed). Since $v_k$ must have a partner in a perfect matching, the range of possible values for $N_k(\tau,f)$ is from $1$ to $d_k$. By definition of conditional entropy, and Conditional maximality of the uniform, we have that for each fixed $\tau \in S_n$
\begin{eqnarray*}
H(f(\tau_k)|f(\tau_1),\ldots,f(\tau_{k-1})) & \leq & \sum_{i=1}^{d_k} \Pr(N_k(\tau,f)=i) \log i  \\
& = & \sum_{i=1}^{d_k} \log i \sum_{f \in {\mathcal M}_{\rm perf}(G)} \frac{|\{f \in {\mathcal M}_{\rm perf}(G):N_k(\tau,f)=i\}|}{|{\mathcal M}_{\rm perf}(G)|},
\end{eqnarray*}
in the first line above returning to viewing $f$ as a uniformly chosen element of ${\mathcal M}_{\rm perf}(G)$, and unpacking this probability in the second line.

An upper bound for the summand in (\ref{eq-averagechainrule}) corresponding to $k$ is now
\begin{equation} \label{eq-Rad-Breg}
\sum_{i=1}^{d_k} \log i \left(\frac{1}{n!|{\mathcal M}_{\rm perf}(G)|}\sum_{f \in {\mathcal M}_{\rm perf}(G)} \sum_{\tau \in S_n} |\{f \in {\mathcal M}_{\rm perf}(G):N_k(\tau,f)=i\}|\right).
\end{equation}
Here is where the power of averaging over all $\tau \in S_n$ comes in. For each fixed $f \in {\mathcal M}_{\rm perf}(G)$, as $\tau$ runs over $S_n$, $N_k(\tau,f)$ is equally likely to take on each of the values $1$ through $d_k$, since $N_k(\tau,f)$ depends only on the position of $f(k)$ in $\tau$, relative to the positions of the other indices $i$ such that $v_kw_i \in E(G)$ (if $f(k)$ is the earliest neighbor of $k$ used by $\tau$, which happens with probability $1/d_k$ for uniformly chosen $\tau \in S_n$, then $N_k(\tau,f)=d_k$; if it is the second earliest, which again happens with probability $1/d_k$, then $N_k(\tau,f)=d_k-1$, and so on). So the sum in (\ref{eq-Rad-Breg}) becomes
$$
\sum_{i=1}^{d_k} \log i \left(\frac{1}{n!|{\mathcal M}_{\rm perf}(G)|}\sum_{f \in {\mathcal M}_{\rm perf}(G)} \frac{n!}{d_k}\right) = \frac{\log d_k!}{d_k},
$$
and inserting into (\ref{eq-averagechainrule}) we obtain
$$
H(X) \leq \sum_{k=1}^n  \frac{\log d_k!}{d_k}
$$
as required.

\subsection{Alon and Friedland's proof of the Kahn-Lov\'asz theorem} \label{subsec-AlonFriedland}

Alon and Friedland's idea is to relate ${\mathcal M}_{\rm perf}(G)$ to the permanent of the adjacency matrix ${\rm Adj}(G)=(a_{ij})$ of $G$. This is the $2n$ by $2n$ matrix with
$$
a_{ij} = \left\{\begin{array}{ll}
1 & \mbox{if $v_iv_j \in E$} \\
0 & \mbox{otherwise.}
\end{array}\right.
$$

An element of ${\mathcal M}_{\rm perf}(G) \times {\mathcal M}_{\rm perf}(G)$ is a pair of perfect matchings. The union of these perfect matchings is a collection of isolated edges (the edges in common to both matchings), together with a collection of disjoint even cycles, that covers the vertex set of the graph. For each such subgraph of $G$ (call it an {\em even cycle cover}), to reconstruct the pair of matchings from which it arose we have to make an arbitrary choice for each even cycle, since there are two ways of writing an even cycle as an ordered union of matchings. It follows that
$$
|{\mathcal M}_{\rm perf}(G) \times {\mathcal M}_{\rm perf}(G)|=\sum_{S} 2^{c(S)}
$$
where the sum is over all even cycle covers $S$ of $G$ and $c(S)$ counts the number of even cycles in $S$.

On the other hand, any permutation $\sigma$ contributing to ${\rm perm}({\rm Adj}(G))$ breaks into disjoint cycles each of length at least $2$, with the property that for each such cycle $(v_{i_1}, \ldots, v_{i_k})$ we have $v_{i_1}v_{i_2}, v_{i_2}v_{i_3}, \ldots, v_{i_k}v_{i_1} \in E$. So such $\sigma$ is naturally associated with a collection of isolated edges (the cycles of length $2$), together with a collection of disjoint cycles (some possibly of odd length), that covers the vertex set of the graph. For each such subgraph of $G$ (call it a {\em cycle cover}), to reconstruct the $\sigma$ from which it arose we have to make an arbitrary choice for each cycle, since there are two ways of orienting it. It follows that
$$
{\rm perm}({\rm Adj}(G))=\sum_{S} 2^{c(S)}
$$
where the sum is over all cycle covers $S$ of $G$ and $c(S)$ counts the number of cycles in $S$.

It is clear that $|{\mathcal M}_{\rm perf}(G) \times {\mathcal M}_{\rm perf}(G)| \leq {\rm perm}({\rm Adj}(G))$ since there are at least as many $S$'s contributing to the second sum as the first, and the summands are identical for $S$'s contributing to both. Applying Br\'egman's theorem to the right-hand side, and taking square roots, we get
$$
|{\mathcal M}_{\rm perf}(G)| \leq \prod_{i=1}^{2n} \left(d_i!\right)^\frac{1}{2d_i}.
$$

\section{Counting proper colorings of a regular graph} \label{sec-counting colorings}

\subsection{Introduction to the problem} \label{subsec-homsintro}

A {\em proper $q$-coloring} (or just {\em $q$-coloring}) of $G$ is a function from the vertices of $G$ to $\{1, \ldots, q\}$ with the property that adjacent vertices have different images. We write $c_q(G)$ for the number of $q$-colorings of $G$.

The following is a natural extremal enumerative question: for a family ${\mathcal G}$ of graphs, which $G \in {\mathcal G}$ maximizes $c_q(G)$? For example, for the family of $n$-vertex, $m$-edge graphs this question was raised independently by Wilf \cite{BenderWilf,Wilf} (who encountered it in his study of the running time of a backtracking coloring algorithm) and Linial \cite{Linial} (who encountered the {\em minimization} question in his study of the complexity of determining whether a given function on the vertices of a graph is in fact a proper coloring).
Although it has only been answered completely in some very special cases many partial results have been obtained (see \cite{LohPikhurkoSudakov} for a good history of the problem).

The focus of this section is the family of $n$-vertex $d$-regular graphs with $d \geq 2$ (the case $d=1$ being trivial). In Section \ref{subsec-GT} we explain an entropy proof of Galvin and Tetali \cite{GalvinTetali-weighted} of the following.
\begin{thm} \label{thm-GTcols}
For $d \geq 2$, $n \geq d+1$ and $q \geq 2$, if $G$ is any $n$-vertex $d$-regular bipartite graph then
$$
c_q(G) \leq c_q(K_{d,d})^\frac{n}{2d}.
$$
\end{thm}
Notice that this upper bound is tight in the case when $2d|n$, being achieved by the disjoint union of $n/2d$ copies of $K_{d,d}$.

Theorem \ref{thm-GTcols} is a special case of a more general result concerning graph homomorphisms. A {\em homomorphism} from $G$ to a graph $H$ (which may have loops) is a map from vertices of $G$ to vertices of $H$ with adjacent vertices in $G$ being mapped to adjacent vertices in $H$. Homomorphisms generalize $q$-colorings (if $H=K_q$ then the set of homomorphisms to $H$ is in bijection with the set of $q$-colorings of $G$) as well as other graph theory notions, such as independent sets. A {\em independent set} in a graph is a set of pairwise non-adjacent vertices; notice that if $H=H_{\rm ind}$ is the graph on two adjacent vertices with a loop at exactly one of the vertices, then a homomorphism from $G$ to $H$ may be identified, via the preimage of the unlooped vertex, with an independent set in $G$. The main result from \cite{GalvinTetali-weighted} is the following generalization of Theorem \ref{thm-GTcols}. Here we write ${\rm hom}(G,H)$ for the number of homomorphisms from $G$ to $H$.
\begin{thm} \label{thm-GThoms}
For $d \geq 2$, $n \geq d+1$ and any finite graph $H$ (perhaps with loops, but without multiple edges), if $G$ is any $n$-vertex $d$-regular bipartite graph then
$$
{\rm hom}(G,H) \leq {\rm hom}(K_{d,d},H)^\frac{n}{2d}.
$$
\end{thm}
The proof of Theorem \ref{thm-GThoms} is virtually identical to that of Theorem \ref{thm-GTcols}; to maintain the clarity of the exposition we just present in the special case of coloring. (Recently Lubetzky and Zhao \cite{LubetskyZhao} gave a proof of Theorem \ref{thm-GThoms} that uses a generalized H\"{o}lder's inequality in place of entropy.)

The inspiration for Theorems \ref{thm-GTcols} and \ref{thm-GThoms} was the special case of enumerating independent sets ($H=H_{\rm ind})$. In what follows we use $i(G)$ to denote the number of independent sets in $G$. Alon \cite{Alon} conjectured that for all $n$-vertex $d$-regular $G$,
$$
i(G) \leq i(K_{d,d})^{n/2d} = (2^{d+1}-1)^{n/2d} = 2^{n/2 + n(1+o(1))/2d}
$$
(where here and in the rest of this section $o(1) \rightarrow 0$ as $d \rightarrow \infty$),
and proved the weaker bound $i(G) \leq 2^{n/2 + Cn/d^{1/10}}$ for some absolute constant $C>0$.

The sharp bound was proved for {\em bipartite} $G$ by Kahn \cite{Kahn}, but it was a while before a bound for general $G$ was obtained that came close to $i(K_{d,d})^{n/2d}$ in the second term of the exponent; this was Kahn's (unpublished) bound $i(G)\leq 2^{n/2 + n(1+o(1))/d}$.
This was improved to $i(G) \leq 2^{n/2 + n(1+o(1))/2d}$ by Galvin
\cite{Galvin-asymptot-ind-count}. Finally Zhao \cite{Zhao} deduced the exact bound for general $G$ from the bipartite case.

A natural question to ask is what happens to the bounds on number of $q$-colorings and homomorphism counts, when we relax to the family of general (not necessarily bipartite) $n$-vertex, $d$-regular graphs; this question remains mostly open, and is discussed in Section \ref{subsec-homs problems}. We will mention one positive result here. By a modification of the proof of Theorems \ref{thm-GTcols} and \ref{thm-GThoms} observed by Kahn, the following can be obtained. (See \cite{MadimanTetali} for a more general statement.)
\begin{thm} \label{thm-non-bip-homs}
Fix $d \geq 2$, $n \geq d+1$ and any finite graph $H$ (perhaps with loops, but without multiple edges), and let $G$ be an $n$-vertex $d$-regular graph (not necessarily bipartite). Let $<$ be a total order on the vertices of $G$, and write $p(v)$ for the number of neighbors $w$ of $v$ with $w < v$. Then
$$
{\rm hom}(G,H) \leq \prod_{v \in V(G)}{\rm hom}(K_{p(v),p(v)},H)^\frac{1}{d}.
$$
In particular,
$$
c_q(G) \leq \prod_{v \in V(G)} c_q(K_{p(v),p(v)})^\frac{1}{d}.
$$
\end{thm}
Notice that Theorem \ref{thm-non-bip-homs} implies Theorems \ref{thm-GTcols} and \ref{thm-GThoms}. Indeed, if $G$ is bipartite with color classes ${\mathcal E}$ and ${\mathcal O}$, and we take $<$ to be an order that puts everything in ${\mathcal E}$ before everything in ${\mathcal O}$, then $p(v)=0$ for each $v \in {\mathcal E}$ (and so the contribution to the product from these vertices is $1$) and $p(v)=d$ for each $v \in {\mathcal O}$ (and so the contribution to the product from each of these vertices is ${\rm hom}(K_{d,d},H)^{1/d}$). Noting that $|{\mathcal O}|=n/2$, the right hand side of the first inequality in Theorem \ref{thm-non-bip-homs} becomes ${\rm hom}(K_{d,d},H)^{n/2d}$.

We prove Theorem \ref{thm-non-bip-homs} in Section \ref{subsec-nonbip}. To appreciate the degree to which the bound differs from those of Theorems \ref{thm-GTcols} and \ref{thm-GThoms} requires understanding ${\rm hom}(K_{p(v),p(v)},H)$, which would be an overly long detour, so we content ourselves now with understanding the case $H=K_3$ (proper $3$-colorings). An easy inclusion-exclusion argument gives
\begin{equation} \label{eq-bip-colo-count}
c_q(K_{d,d})^{\frac{n}{2d}} = \left(6(2^d -1)\right)^\frac{n}{2d} = 2^\frac{n}{2}6^{\frac{n}{d}\left(\frac{1}{2}+o(1)\right)}.
\end{equation}
Theorem \ref{thm-non-bip-homs}, on the other hand,  says that for all $n$-vertex, $d$-regular $G$,
\begin{eqnarray}
c_3(G) & \leq & \prod_{v \in V(G)} c_3(K_{p(v),p(v)})^\frac{1}{d} \nonumber \\
& \leq & \prod_{v \in V(G)} \left(6 \times 2^{p(v)}\right)^\frac{1}{d} \nonumber \\
& = & 2^\frac{\sum_{v} p(v)}{d} 6^\frac{n}{d} \nonumber \\
& = & 2^\frac{n}{2} 6^\frac{n}{d}, \label{eq-non-bip-col-count}
\end{eqnarray}
the last equality use the fact that each edge in $G$ is counted exactly once in $\sum_{v} p(v)$. Comparing (\ref{eq-non-bip-col-count}) with (\ref{eq-bip-colo-count}) we see that in the case of proper 3-coloring, the cost that Theorem \ref{thm-non-bip-homs} pays in going from bipartite $G$ to general $G$ is to give up a constant fraction in the second term in the exponent of $c_3(G)$; there is a similar calculation that can be performed for other $H$.

The bound in (\ref{eq-non-bip-col-count}) has recently been improved \cite{Galvin-asymptot-col-count}: we now know that for all $n$-vertex, $d$-regular $G$,
\begin{equation} \label{eq-bip-colo-count2}
c_3(G) \leq 2^\frac{n}{2} 6^{\frac{n}{2d}\left(\frac{1}{2}+o(1)\right)}.
\end{equation}
(This still falls short of (\ref{eq-bip-colo-count}), since there the $o(1)$ term is negative, whereas in (\ref{eq-bip-colo-count2}) it is positive). The proof does not use entropy, and we don't discuss it any further here.

\subsection{A tight bound in the bipartite case} \label{subsec-GT}

Here we prove Theorem \ref{thm-GTcols}, following an approach first used by Kahn \cite{Kahn} to enumerate independent sets. Let $G$ be a $n$-vertex, $d$-regular bipartite graph, with color classes ${\mathcal E}$ and ${\mathcal O}$ (both of size $n/2$). Let $X$ be a uniformly chosen proper $q$-coloring of $G$. We may view $X$ as a vector $(X_{\mathcal E},X_{\mathcal O})$ (with $X_{\mathcal E}$, for example, being the restriction of the coloring to ${\mathcal E}$). By the Chain rule,
\begin{equation} \label{eq-homs1}
H(X) = H(X_{\mathcal E}) + H(X_{\mathcal O}|X_{\mathcal E}).
\end{equation}
Each of $X_{\mathcal E}, X_{\mathcal O}$ may themselves be viewed as vectors, $(X_v:v \in {\mathcal E})$ and $(X_v:v \in {\mathcal O})$ respectively, where $X_v$ indicates the restriction of the random coloring to vertex $v$. We bound $H(X_{\mathcal O}|X_{\mathcal E})$ by Conditional subadditivity, with (\ref{eq-homs2}) following from Dropping conditioning:
\begin{eqnarray}
H(X_{\mathcal O}|X_{\mathcal E}) & \leq & \sum_{v \in {\mathcal O}} H(X_v|X_{\mathcal E}) \nonumber \\
& \leq & \sum_{v \in {\mathcal O}} H(X_v|X_{N(v)}), \label{eq-homs2}
\end{eqnarray}
where $N(v)$ denotes the neighborhood of $v$.

We bound $H(X_{\mathcal E})$ using Shearer's lemma, taking ${\mathcal F} = \{N(v): v \in {\mathcal O}\}$. Since $G$ is $d$-regular, each $v \in E$ appears in exactly $d$ elements of ${\mathcal F}$, and so
\begin{equation} \label{eq-homs3}
H(X_{\mathcal E}) \leq \frac{1}{d} \sum_{v \in {\mathcal O}} H(X_{N(v)}).
\end{equation}
Combining (\ref{eq-homs3}) and (\ref{eq-homs2}) with (\ref{eq-homs1}) yields
\begin{equation} \label{eq-homs4}
H(X) \leq \frac{1}{d} \sum_{v \in {\mathcal O}} \left(H(X_{N(v)}) + dH(X_v|X_{N(v)})\right).
\end{equation}
Notice that the left-hand side of (\ref{eq-homs4}) concerns a global quantity (the entropy of the coloring of the whole graph), but the right-hand side concerns a local quantity (the coloring in the neighborhood of a single vertex).

We now focus on the summand on the right-hand side of (\ref{eq-homs4}) for a particular $v \in {\mathcal O}$. For each possible assignment $C$ of colors to $N(v)$, write $p(C)$ for the probability of the event $\{X_{N(v)}=C\}$. By the definition of entropy (and conditional entropy),
\begin{eqnarray}
H(X_{N(v)}) + dH(X_v|X_{N(v)}) & = & \sum_C p(C)\log \frac{1}{p(C)} + d\sum_C p(C) H(X_v|\{X_{N(v)}=C\}) \nonumber \\
& = & \sum_C p(C)\left(\log \frac{1}{p(C)} + dH(X_v|\{X_{N(v)}=C\})\right). \label{eq-homs5}
\end{eqnarray}
Let $e(C)$ be the number of ways of properly coloring $v$, given that $N(v)$ is colored $C$. Using Conditional maximality of the uniform we have
$$
H(X_v|\{X_{N(v)}=C\}) \leq \log e(C),
$$
and so from (\ref{eq-homs5}) we get
\begin{eqnarray}
H(X_{N(v)}) + dH(X_v|X_{N(v)}) & \leq & \sum_C p(C)\left(\log \frac{1}{p(C)} + d \log e(C)\right) \nonumber \\
& = & \sum_C p(C) \log \left(\frac{e(C)^d}{p(C)}\right) \nonumber \\
& \leq & \log \left(\sum_C e(C)^d\right), \label{eq-homs6}
\end{eqnarray}
with (\ref{eq-homs6}) an application of Jensen's inequality.

But now notice that the right-hand side of (\ref{eq-homs6}) is exactly $c_q(K_{d,d})$: to properly $q$-color $K_{d,d}$ we first choose an assignment of colors $C$ to one of the two color classes, and then for each vertex of the other class (independently) choose a color from $e(C)$. Combining this observation with (\ref{eq-homs4}), we get
\begin{eqnarray}
H(X) & \leq & \frac{1}{d} \sum_{v \in {\mathcal O}} \log c_q(K_{d,d}) \nonumber \\
& = & \log c_q(K_{d,d})^{n/2d}. \label{eq-homs7}
\end{eqnarray}
Finally, observing that since $X$ is a uniform $q$-coloring of $G$, and so $H(X) = \log c_q(G)$, we get from (\ref{eq-homs7}) that
$$
c_q(G) \leq c_q(K_{d,d})^{n/2d}.
$$

\subsection{A weaker bound in the general case} \label{subsec-nonbip}

We just prove the bound on $c_q(G)$, with the proof of the more general statement being almost identical.  As before, $X$ is a uniformly chosen proper $q$-coloring of $G$, which we view as the vector $(X_v:v \in V(G))$.

Let $<$ be any total order on $V(G)$, and write $P(v)$ for the set of neighbors $w$ of $v$ with $w < v$ (so $|P(v)|=p(v)$). Let ${\mathcal F}$ be the family of subsets of $V(G)$ that consists of one copy of $P(v)$ for each $v \in V(G)$, and $p(v)$ copies of $\{v\}$. Notice that each $v \in V(G)$ appears in exactly $d$ members of ${\mathcal F}$.

We bound $H(X_v:v \in V(G))$ using Kahn's conditional form of Shearer's lemma (Lemma \ref{lemma-conditional Shearer 2}). We have
\begin{eqnarray}
H(X_v:v \in V(G)) & \leq & \frac{1}{d} \sum_{F \in {\mathcal F}} H(X_F|\{X_i:i \prec F\}) \nonumber \\
& = & \frac{1}{d} \sum_{v \in V(G)} \left(H(X_{P(v)}|\{X_i:i < P(v)\}) + p(v)H(X_v|\{X_i:i < v\})\right) \nonumber \\
& \leq & \frac{1}{d} \sum_{v \in V(G)} \left(H(X_{P(v)}) + p(v)H(X_v|X_{P(v)})\right). \label{eq-homsnonbib1}
\end{eqnarray}
In (\ref{eq-homsnonbib1}) we have used Dropping conditioning on both entropy terms on the right-hand side. That
$$
H(X_{P(v)}) + p(v)H(X_v|X_{P(v)}) \leq c_q(K_{p(v),p(v)})
$$
(completing the proof) follows exactly as in the bipartite case (Section \ref{subsec-GT}) (just replace all ``$N$'''s there with ``$P$'''s).

\section{Open problems} \label{sec-open problems}

\subsection{Counting matchings} \label{subsec-matching open}

A natural direction in which to extend Br\'egman's theorem is to consider arbitrary matchings in $G$, rather than perfect matchings. For this discussion, we focus exclusively on the case of $d$-regular $G$ on $2n$ vertices, with $d|2n$. Writing $K(n,d)$ for the disjoint union of $n/d$ copies of $K_{d,d}$, we can restate Br\'egman's theorem as the statement that
\begin{equation} \label{eq-Bregman}
|{\mathcal M}_{\rm perf}(G)| \leq d!^\frac{n}{d} = |{\mathcal M}_{\rm perf}(K(n,d))|
\end{equation}
for bipartite $d$-regular $G$ on $2n$ vertices, and the Kahn-Lov\'asz theorem as the statement that (\ref{eq-Bregman}) holds for arbitrary $d$-regular $G$ on $2n$ vertices.

Do these inequalities continue to hold if we replace ${\mathcal M}_{\rm perf}(G)$ with ${\mathcal M}(G)$, the collection of all matchings (not necessarily perfect) in $G$?
\begin{conj} \label{conj-allmatch}
For bipartite $d$-regular $G$ on $2n$ vertices (or for arbitrary $d$-regular $G$ on $2n$ vertices),
$$
|{\mathcal M}(G)| \leq |{\mathcal M}(K(n,d))|.
$$
\end{conj}
Here
the heart of the matter is the bipartite case: the methods of Alon and Friedland discussed in Section \ref{subsec-AlonFriedland} can be modified to show that the bipartite case implies the general case.

Friedland, Krop and Markstr\"om \cite{FriedlandKropMarkstrom} have proposed an even stronger conjecture, the Upper Matching conjecture. For each $0 \leq t \leq n$, write ${\mathcal M}_t(G)$ for the number of matchings in $G$ of size $t$ (that is, with $t$ edges).
\begin{conj} \label{conj-upmatch} (Upper Matching conjecture)
For bipartite $d$-regular $G$ on $2n$ vertices (or for arbitrary $d$-regular $G$ on $2n$ vertices), and for all $0 \leq t \leq n$,
$$
|{\mathcal M}_t(G)| \leq |{\mathcal M}_t(K(n,d))|.
$$
\end{conj}
For $t=n$ this is Br\'egman's theorem (in the bipartite case) and the Kahn-Lov\'asz theorem (in the general case). For $t=0, 1$ and $2$ it is trivial in both cases. Friedland, Krop and Markstr\"om \cite{FriedlandKropMarkstrom} have verified the conjecture (in the bipartite case) for $t=3$ and $4$. For $t=\alpha n$ for $\alpha \in [0,1]$, asymptotic evidence in favor of the conjecture was provided first by Carroll, Galvin and Tetali \cite{CarrollGalvinTetali} and then (in a stronger form) by Kahn and Ilinca \cite{KahnIlinca}. We now briefly discuss this latter result.

Set $t=\alpha n$, where $\alpha \in (0,1)$ is fixed, and we restrict our attention to those $n$ for which $\alpha n$ is an integer. The first non-trivial task in this range is to determine the asymptotic behavior of $|{\mathcal M}_{\alpha n}(K(n,d))|$ in $n$ and $d$. To do this we start from the identity
$$
|{\mathcal M}_{\alpha n}(K(n,d))| = \sum_{a_1, \ldots a_{n/d}:\atop{0 \leq a_i \leq d,~
\sum_i a_i = \alpha n}} \prod_{i=1}^{n/d} {d \choose a_i}^2 a_i!
$$
Here the $a_i$'s are the sizes of the intersections of the matching
with each of the components of $K(n,d)$, and the term ${d \choose
a_i}^2 a_i!$ counts the number of matchings of size $a_i$ in a
single copy of $K_{d,d}$. (The binomial term represents the choice
of $a_i$ endvertices for the matching from  each color class,
and the factorial term tells us how many ways there are to pair the
endvertices from each class to form a matching.) Considering only those sequences $(a_1, \ldots a_{n/d})$ in which each $a_i$ is either $\lfloor \alpha d \rfloor$ or $\lceil \alpha d \rceil$, we get
\begin{equation} \label{inq-match_lb}
\log |{\mathcal M}_{\alpha n}(K(n,d))| = n\left(\alpha\log d+2H(\alpha)+
\alpha\log\left(\frac{\alpha}{e}\right)+\Omega_\alpha\left(\frac{\log d}{d}\right)\right),
\end{equation}
where $H(\alpha)-\alpha\log\alpha-(1-\alpha)\log(1-\alpha)$ is the binary entropy function.
The detailed analysis appears in \cite{CarrollGalvinTetali}. Using a refinement of Radhakrishnan's approach to Br\'egman's theorem, Kahn and Ilinca \cite{KahnIlinca}
give an upper bound on $\log |{\mathcal M}_{\alpha n}(G)|$ for arbitrary $d$-regular $G$ on $2n$ vertices that agrees with (\ref{inq-match_lb}) in the first two terms:
$$
\log |{\mathcal M}_{\alpha n}(G)| \leq n\left(\alpha\log d+2H(\alpha)+
\alpha\log\left(\frac{\alpha}{e}\right)+o(d^{-1/4})\right).
$$

\subsection{Counting homomorphisms} \label{subsec-homs problems}

Wilf \cite{Wilf} and Linial \cite{Linial} asked which graph on $n$ vertices and $m$ edges maximizes the number of proper $q$-colorings, for each $q$, and similar questions can be asked for other instances of homomorphisms and family of graphs. We will not present a survey of the many questions that have been asked in this vein (and in only a few cases answered); the interested reader might consult Cutler's article \cite{Cutler-survey}.

One question we will highlight comes directly from the discussion in Section \ref{sec-counting colorings}.
\begin{question} \label{question-homs}
Fix $n$, $d$ and $H$. Which $n$-vertex, $d$-regular graph $G$ maximizes ${\rm hom}(G,H)$, and what is the maximum value?
\end{question}
This question has been fully answered for only very few triples $(n,d,H)$. For example, Zhao \cite{Zhao} resolved the question for all $n$ and $d$ in the case where $H$ is the graph encoding independent sets as graph homomorphisms (as discussed in Section \ref{subsec-homsintro}), and he generalized his approach in \cite{Zhao2} to find infinitely many $H$ such that for every $n$ and every $d$,
\begin{equation} \label{eq-homs-bound}
{\rm hom}(G,H) \leq {\rm hom}(K_{d,d},H)^\frac{n}{2d}.
\end{equation}
Galvin and Tetali \cite{GalvinTetali-weighted}, having established (\ref{eq-homs-bound}) for {\em every} $H$ when $G$ is bipartite, conjectured that (\ref{eq-homs-bound}) should still hold for every $H$ when the biparticity assumption on $G$ is dropped, but this conjecture turned out to be false, as did the modified conjecture (from \cite{Galvin2}) that for every $H$ and $n$-vertex, $d$-regular $G$, we have
$$
{\rm hom}(G,H) \leq \max \left\{ {\rm hom}(K_{d,d},H)^\frac{n}{2d}, {\rm hom}(K_{d+1},H)^\frac{n}{d+1}\right\}.
$$
(Sernau \cite{Sernau} has very recently found counterexamples).

While we have no conjecture as to the answer to Question \ref{question-homs} in general, we mention a few specific cases where we do venture guesses.
\begin{conj} \label{conjecture-Kq}
For every $n, d$ and $q$, if $G$ is an $n$-vertex, $d$-regular graph then
$$
{\rm hom}(G,K_q) \leq {\rm hom}(K_{d,d},K_q)^\frac{n}{2d}
$$
(or, in the language of Section \ref{sec-counting colorings}, $c_q(G) \leq c_q(K_{d,d})^{n/2d}$).
\end{conj}
This is open for all $q \geq 3$. Galvin \cite{Galvin2} and Zhao \cite{Zhao2} have shown that it is true when $q=q(n,d)$ is sufficiently large (the best known bound is currently $q > 2\binom{nd/2}{4}$), but neither proof method seems to say anything about constant $q$; see \cite{Galvin-asymptot-col-count} for the best approximate results to date for constant $q$.

Our next conjecture seems ripe for an entropy attack. It concerns the graph $H_{\rm WR}$, the complete fully looped graph on three vertices with a single edge (not a loop) removed (equivalently, $H_{\rm WR}$ is the complete looped path on 3 vertices). Homomorphisms to $H_{\rm WR}$ encode configurations in the {\em Widom-Rowlinson} model from statistical physics \cite{RowlinsonWidom}.
\begin{conj} \label{conjecture-WR}
For every $n$ and $d$, if $G$ is an $n$-vertex, $d$-regular graph then
\begin{equation} \label{eq-WR}
{\rm hom}(G,H_{\rm WR}) \leq {\rm hom}(K_{d+1},H_{\rm WR})^\frac{n}{d+1}.
\end{equation}
\end{conj}
A weak result in the spirit of Conjecture \ref{conjecture-WR} appears in \cite{Galvin2}: for every $n$, $d$ and $q$ satisfying $q>2^{nd/2 + n/2 -1}$, if $H_{\rm WR}^q$ is the complete looped graph on $q$ vertices with a single edge (not a loop) removed, and if $G$ is an $n$-vertex, $d$-regular graph, then (\ref{eq-WR}) (with $H_{\rm WR}^q$ in place of $H_{\rm WR}$) holds.

\medskip

Our final conjecture concerns the number of independent sets of a fixed size in a regular graph, and is due to Kahn \cite{Kahn}. We write $i_t(G)$ for the number of independent sets of size $t$ in a graph $G$.
\begin{conj} \label{conj-indfixed}
For $d$-regular bipartite $G$ on $n$ vertices (or for arbitrary $d$-regular $G$ on $n$ vertices) with $2d|n$, and for all $0 \leq t \leq n$,
$$
i_t(G) \leq i_t(K(n,2d))
$$
(where recall $K(n,2d)$ is the disjoint union of $n/2d$ copies of $K_{d,d}$).
\end{conj}
This is an independent set analog of the Upper Matching conjecture (Conjecture \ref{conj-upmatch}). It is only known (in the bipartite case) for $t \leq 4$, a recent result of Alexander and Mink \cite{AlexanderMink}.

\section{Bibliography of applications of entropy to combinatorics} \label{sec-other papers}

Here we give a brief bibliography of the use of entropy in combinatorial enumeration problems. It is by no means comprehensive, and the author welcomes any suggestions for additions. Entries are presented chronologically (and alphabetically within each year).

\begin{itemize}

\item Erd\H{o}s \& R\'enyi, On two problems of information theory (1963) \cite{ErdosRenyi}: The first combinatorial application of entropy, this paper gives a lower bound on the size of the smallest distinguishing family of a set.

\item Pippenger, An information-theoretic method in combinatorial theory (1977) \cite{Pippinger}: Gives a lower bound on the size of the smallest distinguishing family, and a lower bound on the sum of the sizes of complete bipartite graphs that cover a complete graph.

\item Chung, Frankl, Graham \& Shearer, Some intersection theorems for ordered sets and
     graphs (1986) \cite{ChungFranklGrahamShearer}: Introduces Shearer's lemma, and uses it to bound the sizes of some families of sets subject to intersection restrictions.

\item Radhakrishnan, An entropy proof of Br\'egman's theorem (1997) \cite{Radhakrishnan}: Gives a new proof of Br\'egman's theorem on the maximum permanent of a 0-1 matrix with fixed row sums.

\item Friedgut \& Kahn, On the number of copies of one hypergraph in another (1998) \cite{FriedgutKahn}: Gives near-matching upper and lower bounds on the maximum number of copies of one graph that can appear in another graph with a given number of edges (only the upper bound uses entropy).

\item Kahn \& Lawrenz, Generalized Rank Functions and an Entropy Argument (1999) \cite{KahnLawrenz}: Puts a logarithmically sharp upper bound on the number of rank functions of the Boolean lattice.

\item Pippenger, Entropy and enumeration of Boolean functions (1999) \cite{Pippinger2}: Gives a new proof of an upper bound on the number of antichains in the Boolean lattice.

\item Kahn, An Entropy Approach to the Hard-Core Model on Bipartite Graphs (2001) \cite{Kahn}: Studies the structure of a randomly chosen independent set drawn both from an arbitrary regular bipartite graph, and from the family of hypercubes and discrete even tori, and gives a tight upper bound on the number of independent sets admitted by a regular bipartite graph.

\item Kahn, Range of cube-indexed random walk (2001) \cite{Kahn2}: Answers a question of Benjamini, H\"aggstr\"om and Mossel on the typical range of a labeling of the vertices of a hypercube in which adjacent vertices receive labels differing by one.

\item Kahn, Entropy, independent sets and antichains: A new approach to Dedekind's problem (2001) \cite{Kahn-dedekind}: Gives a new proof of an upper bound on the number of antichains in the Boolean lattice (with entropy entering in in proving the base-case of an induction).

\item Radhakrishnan, Entropy and counting (2003) \cite{Radhakrishnan2}: A survey article.

\item Friedgut, Hypergraphs, Entropy and Inequalities (2004) \cite{Friedgut}: Shows how a generalization of Shearer's lemma has numerous familiar inequalities from analysis as special cases.

\item Galvin \& Tetali, On weighted graph homomorphisms (2004) \cite{GalvinTetali-weighted}: Gives a sharp upper bound on the number of homomorphisms from a regular bipartite graph to any fixed target graph.

\item Friedgut \& Kahn, On the Number of Hamiltonian Cycles in a Tournament (2005) \cite{FriedgutKahn2}: Obtains the to-date best upper bound on the maximum number of Hamilton cycles admitted by a tournament.

\item Johansson, Kahn \& Vu, Factors in random graphs (2008) \cite{JohanssonKahnVu}: Obtains the threshold probability for a random graph to have an $H$-factor, for each strictly balanced $H$. Entropy appears as part of a lower bound on the number of $H$-factors in $G(n,p)$.

\item Carroll, Galvin \& Tetali, Matchings and Independent Sets of a Fixed Size in Regular Graphs (2009) \cite{CarrollGalvinTetali}: Approximates the number of matchings and independent sets of a fixed size admitted by a regular bipartite graph.

\item Cuckler \& Kahn, Entropy bounds for perfect matchings and Hamiltonian cycles (2009) \cite{CucklerKahn}: Puts upper bounds on the number of perfect matchings and Hamilton cycles in a graph.

\item Cuckler \& Kahn, Hamiltonian cycles in Dirac graphs (2009) \cite{CucklerKahn2}: Puts a lower bound on the number of Hamilton cycles in an $n$-vertex graph with minimum degree at least $n/2$.

\item Madiman \& Tetali, Information Inequalities for Joint Distributions, with Interpretations and Applications (2010) \cite{MadimanTetali}: Develops generalizations of subadditivity, and gives applications to counting graph homomorphisms and zero-error codes.

\item Cutler \& Radcliffe, An entropy proof of the Kahn-Lov\'asz theorem (2011) \cite{CutlerRadcliffe}: Gives a new proof of the extension of Br\'egman's theorem to general (non-bipartite) graphs.

\item Kopparty \& Rossman, The homomorphism domination exponent (2011) \cite{KoppartyRossman}: Initiates the study of a quantity closely related to homomorphism counts, called the homomorphism domination exponent.

\item Balister \& Bollob\'as, Projections, entropy and sumsets (2012) \cite{BalisterBollobas}: Explores the connection between entropy inequalities and combinatorial number-theoretic subset-sum inequalities.

\item Engbers \& Galvin, $H$-coloring tori (2012) \cite{EngbersGalvin1}: Obtains a broad structural characterization of the space of homomorphisms from the family of hypercubes and discrete even tori to any graph $H$, and derives long-range influence consequences.

\item Engbers \& Galvin, $H$-colouring bipartite graphs (2012) \cite{EngbersGalvin2}: Studies the structure of a randomly chosen homomorphism from an arbitrary regular bipartite graph to an arbitrary graph.

\item Madiman, Marcus \& Tetali, Entropy and set cardinality inequalities for partition-determined functions and application to sumsets (2012) \cite{MadimanMarcusTetali}: Explores the connection between entropy inequalities and combinatorial number-theoretic subset-sum inequalities.

\item Ilinca \& Kahn, Asymptotics of the upper matching conjecture (2013) \cite{KahnIlinca}: Gives the to-date best upper bounds on the number of matchings of a fixed size admitted by a regular bipartite graph.

\item Ilinca \& Kahn, Counting Maximal Antichains and Independent Sets (2013) \cite{KahnIlinca2}: Puts upper bounds on the number of maximal antichains in the $n$-dimensional
    Boolean algebra and on the numbers of maximal independent sets in the covering
    graph of the $n$-dimensional hypercube.

\item Balogh, Csaba, Martin \& Pluhar, On the path separation number of graphs (preprint) \cite{BaloghCsabaMartinPluhar}: Gives a lower bound on the path separation number of a graph.

\end{itemize}


\begin{thebibliography}{99}

\bibitem{AlexanderMink}
J. Alexander and T. Mink, A new method for enumerating independent sets of a fixed size in general graphs, arXiv:1308.3242.

\bibitem{Alon2}
Alon, N, On the number of subgraphs of prescribed type of graphs with a given
number of edges, {\em Israel Journal of Mathematics} {\bf 38} (1981), 116--130.

\bibitem{Alon}
N. Alon, Independent sets in regular graphs and sum-free subsets of finite groups, {\em Israel J. Math.} {\bf 73} (1991), 247--256.

\bibitem{AlonFriedland}
N. Alon and S. Friedland, The maximum number of perfect matchings
in graphs with a given degree sequence, {\em Electron. J. Combin.} {\bf 15} (2008),
\#N13.

\bibitem{AlonSpencer}
N. Alon and J. Spencer, {\em The
Probabilistic Method}, Wiley, New York, 2000.
%

\bibitem{BalisterBollobas}
P. Balister and B. Bollob\'as, Projections, entropy and sumsets, {\em Combinatorica} {\bf 32} (2012), 125-141.

\bibitem{BaloghCsabaMartinPluhar}
J. Balogh, B. Csaba, R. Martin and A. Pluhar, On the path separation number of graphs, arXiv:1312.1724.

\bibitem{BenderWilf}
E. Bender and H. Wilf, A theoretical analysis of backtracking in the graph coloring problem,
{\em Journal of Algorithms} {\bf 6} (1985), 275--282.



\bibitem{Bregman}
L. Br\'egman, Some properties of nonnegative matrices and their permanents, {\em Soviet
Math. Dokl.} {\bf 14} (1973), 945--949.


\bibitem{CantorMills}
D. Cantor and W. Mills, Determination of a subset from certain combinatorial
properties, {\em Can. J. Math.} {\bf 18} (1966), 42--48.

\bibitem{CarrollGalvinTetali}
T. Carroll, D. Galvin and P. Tetali, Matchings and Independent Sets of a Fixed Size in Regular Graphs, {\em J. Combin. Theory Ser. A} {\bf 116} (2009), 1219--1227.

\bibitem{ChungFranklGrahamShearer} F. Chung, P. Frankl, R. Graham
and J. Shearer, Some intersection theorems for ordered sets and
graphs, {\em J. Combin. Theory Ser. A.} {\bf 48} (1986), 23--37.

\bibitem{CucklerKahn}
B. Cuckler and J. Kahn, Entropy bounds for perfect matchings and Hamiltonian cycles, {\em Combinatorica} {\bf 29} (2009), 327--335.

\bibitem{CucklerKahn2}
B. Cuckler and J. Kahn, Hamiltonian cycles in Dirac graphs, {\em Combinatorica} {\bf 29} (2009), 299--326.

\bibitem{Cutler-survey}
J. Cutler, Coloring graphs with graphs: a survey, {\em Graph Theory Notes N.Y.} {\bf 63} (2012), 7--16.

\bibitem{CutlerRadcliffe}
J. Cutler and A. Radcliffe, An entropy proof of the Kahn-Lov\'asz theorem, {\em Electronic Journal of Combinatorics} {\bf 18} (2011), \#P10.

\bibitem{Egorychev}
G. Egorychev, Permanents, Book in Series of Discrete Mathematics
(in Russian), Krasnoyarsk, SFU, 2007.

\bibitem{EllisFilmusFriedgut}
D. Ellis, Y. Filmus and E. Friedgut, Triangle intersecting families of graphs, {\em Journal of the European Math. Soc.} {\bf  14} (2012), 841--885.

\bibitem{EngbersGalvin1}
J. Engbers and D. Galvin, $H$-coloring tori, {\em J. Combin. Theory Ser. B} {\bf 102} (2012), 1110--1133.

\bibitem{EngbersGalvin2}
J. Engbers and D. Galvin, $H$-colouring bipartite graphs (with J. Engbers), {\em J. Combin. Theory Ser. B} {\bf 102} (2012), 726-–742.

\bibitem{ErdosRenyi}
P. Erd\H{o}s and A. R\'enyi, On two problems of information theory, {\em Publ. Hung. Acad.
Sci.} {\bf 8} (1963), 241--254.

\bibitem{Friedgut}
E. Friedgut, Hypergraphs, Entropy and Inequalities, {\em The American Mathematical Monthly} {\bf 111} (2004), 749--760.

\bibitem{FriedgutKahn}
E. Friedgut and J. Kahn, On the number of copies of one hypergraph in another,
{\em Israel Journal of Mathematics} {\bf 105} (1998), 251--256.

\bibitem{FriedgutKahn2}
E. Friedgut and J. Kahn, On the Number of Hamiltonian Cycles in a Tournament, {\em Combinatorics Probability and Computing} {\bf 14} (2005), 769--781.

\bibitem{Friedland}
S. Friedland, An upper bound for the number of perfect matchings in
graphs, arXiv:0803.0864.

\bibitem{FriedlandKropMarkstrom}
S. Friedland, E. Krop and K. Markstr\"om, On the Number of Matchings
in Regular Graphs, {\em Electron. J. Combin.} {\bf 15} (2008), \#R110.



\bibitem{Galvin-asymptot-col-count}
D. Galvin, Counting colorings of a regular graph, {\em Graphs Combin.}, DOI 10.1007/s00373-013-1403-z.

\bibitem{Galvin-asymptot-ind-count}
D. Galvin, An upper bound for the number of independent sets in regular graphs, {\em Discrete Math.} {\bf 309} (2009), 6635--6640.

\bibitem{Galvin2}
D. Galvin, Maximizing H-colorings of regular graphs, {\em Journal of Graph Theory} {\bf 73} (2013), 66--84.

\bibitem{GalvinTetali-weighted}
D. Galvin and P. Tetali, On weighted graph homomorphisms, DIMACS Series in Discrete Mathematics and Theoretical Computer Science {\bf 63} (2004) {\em
Graphs, Morphisms and Statistical Physics},
97--104.


\bibitem{Han}
T. Han, Nonnegaitive entropy measures for multivariate symmetric correlations, {\em Inform. Contr.} {\bf 36} (1978), 133--156.



\bibitem{JohanssonKahnVu}
A. Johansson, J. Kahn and V. Vu, Factors in random graphs, {\em Random Structures Algorithms} {\bf 33} (2008), 1--28.

\bibitem{KahnIlinca}
L. Ilinca and J. Kahn, Asymptotics of the upper matching conjecture, {\em J. Combin. Theory Ser. A} {\bf 120} (2013), 976--983.

\bibitem{KahnIlinca2}
L. Ilinca and J. Kahn, Counting Maximal Antichains and Independent Sets, {\em Order} {\bf 30} (2013), 427--435.

\bibitem{Kahn}
J. Kahn, An Entropy Approach to the Hard-Core Model on Bipartite Graphs,
{\em Combin. Probab. Comput.} {\bf 10} (2001),
219--237.


\bibitem{Kahn2}
J. Kahn, Range of cube-indexed random walk, {\em Israel J.
Math.} {\bf 124} (2001) 189--201.

\bibitem{Kahn-dedekind}
J. Kahn, Entropy, independent sets and antichains: A new approach to Dedekind's problem, {\em Proc. AMS} {\bf 130} (2001), 371--378.

\bibitem{KahnLawrenz}
J. Kahn and A. Lawrenz, Generalized Rank Functions and an Entropy Argument, {\em Journal of Combinatorial Theory, Series A} {\bf 87} (1999), 398--403.

\bibitem{KoppartyRossman}
S. Kopparty and B. Rossman, The homomorphism domination exponent, {\em European Journal of Combinatorics} {\bf 32} (2011), 1097--1114.

\bibitem{Lindstrom}
B. Lindstr\"om, On a combinatorial problem in number theory, {\em Can. Math. Bull.} {\bf 8} (1965), 477-490.

\bibitem{Linial}
N. Linial, Legal coloring of graphs, {\em Combinatorica} {\bf 6} (1986), 49--54.

\bibitem{LohPikhurkoSudakov}
P.-S. Loh, O. Pikhurko and B. Sudakov, Maximizing the Number of $q$-Colorings,
{\em Proc. London Math. Soc.} {\bf 101} (2010), 655--696.

\bibitem{LoomisWhitney}
L. Loomis and H. Whitney, An inequality related to the isoperimetric inequality, {\em Bull. Amer. Math. Soc.} {\bf 55} (1949), 961-–962.

\bibitem{LubetskyZhao}
E. Lubetsky and Y. Zhao, On replica symmetry of large deviations in random graphs, {\em Random Structures Algorithms}, DOI: 10.1002/rsa.20536.

\bibitem{MadimanMarcusTetali}
M. Madiman, A. Marcus and P. Tetali, Entropy and Set Cardinality Inequalities for Partition-determined Functions and Application to Sumsets, {\em Random Structures \& Algorithms} {\bf 40} (2012), 399--424.

\bibitem{MadimanTetali}
M. Madiman and P. Tetali, Information Inequalities for Joint Distributions, with Interpretations and Applications, {\em IEEE Trans. on Information Theory} {\bf 56} (2010), 2699--2713.

\bibitem{Moser}
L. Moser, The second moment method in combinatorial analysis, in ``Combinatorial
Structures and Their Applications'', 283--384, Gordon and Breach, New York,
1970.

\bibitem{Pippinger}
N. Pippinger, An information-theoretic method in combinatorial theory, {\em Journal of Combinatorial Theory, Series A} {\bf
23} (1977) 99--104.

\bibitem{Pippinger2}
N. Pippenger, Entropy and enumeration of Boolean functions, {\em IEEE Trans. Info.
Th.} {\bf 45} (1999), 2096--2100.

\bibitem{Radhakrishnan2}
J. Radhakrishnan, Entropy and counting, in Computational mathematics, modelling and algorithms (J. C. Misra, editor), Narosa, 2003, 146--168.

\bibitem{Radhakrishnan}
J. Radhakrishnan, An entropy proof of Br\'egman's theorem, {\em J. Combin. Theory, Ser. A} {\bf 77}, (1997), 161--164.
%

\bibitem{RowlinsonWidom}
J. Rowlinson and B. Widom, New Model for the Study of Liquid-Vapor Phase Transitions,
{\em J. Chem. Phys.} {\bf 52} (1970), 1670--1684.

\bibitem{Ross}
S. Ross, {\em A first course in probability}, Pearson Prentice Hall, Upper Saddle
River, 2009.

\bibitem{Schrijver}
A. Schrijver, A short proof of Minc's conjecture, {\em J. Combin. Theory Ser. A} {\bf 25} (1978), 80--83.

\bibitem{Sernau}
L. Sernau, personal communication.

\bibitem{Shannon}
C. Shannon,  A Mathematical Theory of Communication, {\em Bell System Technical Journal} {\bf 27} (3) (1948), 379-–423.

\bibitem{Shapiro}
H. Shapiro, Problem E 1399, {\em The American Mathematical Monthly} {\bf 67} (1960), 82.

\bibitem{TribusMcIrvine}
M. Tribus and E. McIrvine, Energy and Information, {\em Scientific American} {\bf 225} (1971), 179--188.

\bibitem{Valiant}
L. Valiant, The Complexity of Computing the Permanent, {\em Theoretical Computer Science} {\bf 8} (1979), 189--201.

\bibitem{Wilf}
H. Wilf, Backtrack: An $O(1)$ expected time algorithm for the graph coloring problem, {\em Information
Processing Letters} {\bf 18} (1984), 119--121.

\bibitem{Zhao}
Y. Zhao, The Number of Independent Sets in a Regular Graph, {\em
Combin. Probab. Comput.} {\bf 19} (2010), 315--320.

\bibitem{Zhao2}
Y. Zhao, The bipartite swapping trick on graph homomorphisms, {\em SIAM J. Discrete Math} {\bf 25} (2011), 660-680.

\end{thebibliography}
\end{document}